\pdfoutput=1
\documentclass[11pt]{amsart}
\usepackage{amssymb}
\usepackage{epstopdf}
\usepackage{nicefrac}
\usepackage{enumerate}
\usepackage{hyperref}
\usepackage{float}
\usepackage{xcolor}
\usepackage{mathtools}
\usepackage{tabularx}
\usepackage{tikz-cd}
\usepackage{footmisc}
\usepackage{dsfont}
\usepackage{amsthm}
\usepackage{thm-restate}

\usepackage{listings}
\definecolor{codegreen}{rgb}{0,0.6,0}
\definecolor{codegray}{rgb}{0.5,0.5,0.5}
\definecolor{codepurple}{rgb}{0.58,0,0.82}
\definecolor{codestring}{rgb}{0.8,0.3,0.1}
\newcommand{\Z}{\mathbb{Z}}
\newcommand{\F}{\mathbb{F}}
\newcommand{\A}{\mathbb{A}}

\newcommand{\arb}{\operatorname{Arb}_{f,\gamma}}

\input xy
\xyoption {all}

\title{Local-global conjugacy questions for affine extensions}
\author{Dean Wardell}
\date{}


\theoremstyle{definition}
\newtheorem{defn}{Definition}[section]
\newtheorem*{acknowledgement}{Acknowledgement}

\theoremstyle{plain}
\newtheorem{thm}[defn]{Theorem}
\newtheorem{claim}[defn]{Claim}
\newtheorem*{claim*}{Claim}
\newtheorem{cor}[defn]{Corollary}
\newtheorem{question}[defn]{Question}
\newtheorem{prop}[defn]{Proposition}
\newtheorem{lem}[defn]{Lemma}
\newtheorem{conj}[defn]{Conjecture}

\theoremstyle{remark}
\newtheorem{example}[defn]{Example}
\newtheorem{remark}[defn]{Remark}

\newcommand{\p}{\mathcal{P}}
\newcommand{\q}{\mathcal{Q}}
\newcommand{\id}{\operatorname{id}}

\newcommand{\gal}{\operatorname{Gal}}
\newcommand{\sep}{K^{\operatorname{sep}}}
\newcommand{\T}{\operatorname{Aut}(T)}
\newcommand{\Tn}{\operatorname{Aut}(T_n)}
\newcommand{\Tnn}{\operatorname{Aut}(T_{n-1})}
\newcommand{\im}{\operatorname{im}}
\newcommand{\aut}{\operatorname{Aut}}

\newcommand{\one}{\operatorname{id}}
\newcommand{\zero}{\textbf{0}}

\newcommand\isomto{\stackrel{\textstyle\sim}{\smash{\longrightarrow}\rule{0pt}{0.4ex}}}

\begin{document}

\begin{abstract}

    Boston and Jones constructed a probabilistic model, called the Markov model, in order to predict the cycle structures of elements in Galois groups $G_n(f)$ associated to the $n$-th iterate of a quadratic postcritically finite polynomial $f$ over a number field. Goksel refined this model, introducing the 'even' Markov groups $M_n(f)$. These groups conjecturally contain a copy of $G_n(f)$, leading to questions about local-global conjugacies within the larger automorphism group $\Tn$ of the first $n$ levels of the binary rooted tree coming from arboreal representations of the Galois groups.

    While the conjugacy results found by Goksel were restricted to the study of these automorphism groups, we generalise the findings to affine extensions of groups by permutation representations, using a cohomological argument. Furthermore, we provide a counterexample to the main conjecture proposed by Goksel, demonstrating that more work is required to resolve the underlying questions regarding Markov models.
    
\end{abstract}

\maketitle

\section{Introduction}

Let $K$ be a field, consider a polynomial $f\in K[x]$, and let $\gamma\in K$. For each $n\geq 0$, let $f^n$ denote the $n$-th iterate of $f$, where by convention $f^0(x)=x$. Throughout, we assume that for all $n\geq 1$ the polynomial $f^n(x)-\gamma\in K[x]$ is separable. The roots of these iterates in a fixed separable closure $K^{\text{sep}}$ naturally give rise to a $\deg(f)$-ary regular rooted tree $T$ by taking its vertex set to be $$\bigsqcup_{k\geq 0} f^{-k}(\gamma),$$ and connecting vertices $\alpha\in f^{-n}(\gamma)$ and $\beta\in f^{-n-1}(\gamma)$ by an edge if and only if $f(\beta)=\alpha$. Naturally, since elements of the Galois group $\gal(\sep/K)$ commute with (the action of) $f$, we obtain induced actions of $\gal(\sep/K)$ on each preimage set $f^{-n}(\gamma)$ that preserve the adjacency of vertices of $T$. Hence, associated to $f$ and $\gamma$, we obtain an \textit{arboreal Galois representation} $$\arb:\gal(\sep/K)\to \T.$$

Similarly, if for $n\geq 0$ we define $G_n(f)$ to be the Galois group of the splitting field of $f^n(x)-\gamma$ over $K$, then $G_n(f)$ embeds into $\aut(T_n)$, where $T_n$ is the subgraph of $T$ with vertex set $$\bigsqcup_{k=0}^n f^{-k}(\gamma),$$ and where as edge set we take all edges of $T$ that connect vertices in $T_n$. We may now equivalently study the representation $\arb$ via the inverse limit $$\im(\arb)\cong \varprojlim G_n(f).$$ In either case, understanding $\im(\arb)$ as a subgroup of $\T$ is one of the central problems in arithmetic dynamics, see for example \cite{Benedetto,Jones} for an overview. It is expected that, for most polynomials, $\im(\arb)$ has finite index in $\T$, with post-critically finite polynomials being an exceptional case where this index is infinite.

The study of these problems dates back to the 1980s, when Odoni studied prime divisors in polynomial sequences of the form $a_n =f(a_{n-1})$, see for example \cite{Odoni}. When $K$ is a number field, the natural density of such primes within the set of all primes can be related to the sizes of conjugacy classes within $G_n(f)$ through the use of the Chebotarev density theorem, see \cite{StevenhagenLenstra} for a detailed discussion of the theorem. Essentially, determining this density relies on understanding the cycle structures of the associated Frobenius elements $\text{Frob}_\mathfrak{p}$ for various primes $\mathfrak{p}\subseteq \mathcal{O}_K$ of the ring of integers. As the cycle structure of the action of $\text{Frob}_\mathfrak{p}$ on the roots of a polynomial $g$ is given by the degrees of the irreducible factors of the reduction $\overline{g}\in \F_\mathfrak{p}[x]$, the factorisation of reduced polynomials $\overline{f^n}(x)-\overline{\gamma}\in\F_\mathfrak{p}[x]$ for various primes provides information about the Frobenius elements.

Using the density distributions given by the Chebotarev density theorem, Boston and Jones introduced a Markov model in \cite{BostonJonesSettled} for post-critically finite quadratic polynomials $f$ (where one assumes $\gamma = 0$ up to conjugation), in order to predict the behaviour of the factorisations of iterates $f^n$ modulo primes. Given the data of some level $n$, one can predict the densities of any cycle structure at level $n$, which could allow us to create a prediction for the Galois group $G_n(f)$, as well as determine the level $n+1$ data by applying the Markov process as given by the model. While they found an explicit example where the actual cycle data of the Galois group at level 6 did not match the level 6 data obtained from the Markov process, the prediction still appeared to correspond to an actual subgroup of $\aut(T_6)$.

Goksel further investigated this phenomenon in \cite{GokselMarkov}. By restricting to primes $p\equiv 1 \mod 4$, he formally defines the so-called \emph{even} Markov groups $M_n(f)$. These groups refine the Boston-Jones heuristic, providing a different model candidate that conjecturally contains the Galois groups $G_n(f)$.

\begin{conj}[\cite{GokselMarkov}, Conjecture 7.1] \label{Conj: GokselMarkov}
    For all $n\geq 1$, $G_n(f)$ is isomorphic to a subgroup of $M_n(f)$.
\end{conj}

One key observation of \cite{GokselMarkov} was that Conjecture \ref{Conj: GokselMarkov} could be proven by a purely group-theoretic argument inside the automorphism groups $\Tn$, as we explain below. Firstly, we naturally identify $\T=\varprojlim \Tn$, where the quotient maps $$\pi_n:\Tn\to \Tnn$$ are given by restricting the automorphisms to the subgraphs. Define $K_n:=\ker(\pi_n)$. Then, for any pair of subgroups $H,G\leq \Tn$, we consider the following definitions.

\begin{defn}[\cite{GokselMarkov}, Definition 7.2] \label{IntroDef: defn7.2}
    Let $n\geq 0$, let $T_n$ denote the first $n$ levels of the rooted binary tree, and let $H,G\leq \Tn$. We say that $H$ is \textbf{elementwise $K_n$-conjugate into $G$} if for each $h\in H$ there exists some $k_h\in K_n$ such that $k_hhk_h^{-1}\in G$. Similarly, we say that \textbf{$H$ is globally $K_n$-conjugate into $G$} if there exists some $k\in K_n$ such that $kHk^{-1}\subseteq G$.
\end{defn}

\begin{defn}[\cite{GokselMarkov}, Definition 7.3] \label{IntroDef: defn7.3}
    Let $n\geq 0$, let $T_n$ denote the first $n$ levels of the rooted binary tree, and $H,G\leq \Tn$. We say that the property $\p(H,G)$ \textbf{holds} if $H$ is elementwise $K_n$-conjugate into $G$ if and only if $H$ is globally $K_n$-conjugate into $G$.
\end{defn}

Back to the dynamical setting, Goksel noted in \cite{GokselMarkov} that if we assume $G_{n-1}(f)\subseteq M_{n-1}(f)$ for some $n$, then by the construction of the group $M_n(f)$, the lifts (or extensions) of elements $x\in G_{n-1}(f)$ into $G_n(f)$ and $M_n(f)$ should be $K_n$-conjugate. Therefore, $G_n(f)$ becomes elementwise $K_n$-conjugate into $M_n(f)$, and verifying $\p(G_n(f),M_n(f))$ would prove that $G_n(f)$ is conjugate to a subgroup of $M_n(f)$. By induction, Conjecture \ref{Conj: GokselMarkov} holds. As a result, a natural group-theoretic question arose.

\begin{question}[\cite{GokselMarkov}, Question 7.4] \label{Question: GokselMarkov}
    For which subgroups $H,G\leq \Tn$ does $\p(H,G)$ hold?
\end{question}

The property $\p(H,G)$ is known in two generic cases, both first stated, but not proven, in \cite{GokselMarkov}. A proof of the latter statement can be found in \cite{Goksel}.

\begin{thm}[{\cite{GokselMarkov}},Theorem 7.5] \label{Thm: Goksel's proven properties}
    Let $n\geq 1$, and consider subgroups $H,G\leq \Tn$ of the automorphism group of the first $n$ levels of the rooted binary tree. Then $\p(H,G)$ holds in the following two cases:

    \begin{itemize}
        \item[(a)] $|H|=2^{n+1}$ and $H$ contains an element acting transitively on the $n$-th level vertices of the tree $T$.
        \item[(b)] $|H|=|G|$ and $H\cap K_n=\{\one\}$.
    \end{itemize}
\end{thm}

Naturally, one can consider whether the groups $H:=G_n(f)$ and $G:=M_n(f)$ satisfy the hypotheses of Theorem \ref{Thm: Goksel's proven properties}. Unfortunately, these hypotheses appear to be restrictive. For case (a) to apply, we would require the size of the Galois group to be exactly $2^{n+1}$, meaning it must grow exceptionally slow compared to $|\Tn|=2^{2^n-1}$. At the same time, case (b) requires $G_n(f)\cap K_n =\{\one\}$, which by the first isomorphism theorem implies that $G_n(f)\cong G_{n-1}(f)$. For this to be true, the splitting fields of $f^n$ and $f^{n-1}$ need to coincide, which is generally not expected for a given map $f$. Of course, these cases are still valuable, as they could prove a good base for further investigations on the properties $\p(H,G)$.

Based on many computer experiments involving MAGMA, Goksel stated a conjecture following the ideas given throughout the carefully explained examples.

\begin{conj}[\cite{Goksel}, Conjecture 6.1] \label{Conj: Goksel}
    Let $n\geq 0$, let $T_n$ denote the first $n$ levels of the rooted binary tree, and let $H\leq \Tn$ be a group that contains an element acting transitively on $\{1,...,2^n\}$. Then $\p(H,G)$ holds for any $G\leq \Tn$. 
\end{conj}

When restricting ourselves to irreducible iterates of polynomials, it is known that for many polynomials $f$ the groups $G_n(f)$ contain transitive elements, see for example \cite{Odoni}. So, a proof of Conjecture \ref{Conj: Goksel} would immediately resolve Conjecture \ref{Conj: GokselMarkov} for many quadratic polynomials $f$. Sadly, we show by a counterexample that this conjecture is false, though it is uncertain whether the groups we find are realisable as Galois groups and Markov groups. Therefore, if Conjecture \ref{Conj: GokselMarkov} is correct, then another currently undocumented property from either the Galois group or even Markov group forces a stronger relation between local and global conjugation. 

It is natural to ask the same questions for trees $T$ that are $d$-ary instead of binary. Indeed, work on higher degree cases has already started, see for example the work of Martínez in \cite{Martinez} for a study of the cubic case. Of course, since higher degree cases admit more possibilities of postcritical orbits, the Markov model needs to be altered to work. Fortunately, the Chebotarev density theorem is not restricted to quadratic maps, and so the underlying principles of the Markov model remain valid. 

The purpose of this paper is to investigate the properties $\p(H,G)$ in greater generality. Our approach is directly motivated by the behaviour of unicritical polynomials. Namely, if $f$ is a unicritical polynomial of prime degree $p$, its finite-level \emph{geometric monodromy groups} embed, up to global conjugacy, into an $n$-fold iterated wreath product $[C_p]^n$ of $C_p\cong \Z/p\Z$, see for example \cite{AdamsHyde}. In particular, these wreath products take the form $$\F_p^{V_{n-1}}\rtimes [C_p]^{n-1},$$ where $\F_p$ is the field of $p$ elements and $V_{n-1}$ is the set of words of length $n-1$ over the alphabet $\{0,...,p-1\}$. The form of these semidirect products motivates our introduction of \emph{affine extensions} as generalisation of $\Tn$ when $T$ is a binary rooted tree. Affine extensions are of the form $$\F^X\rtimes_\rho P,$$ where $X$ is a finite set, $\F$ is a finite field, $P$ is a finite group, and where $\rho:P\to \aut(\F^X)$ is a permutation representation, defined in Section \ref{Section: Preliminaries}. We note that, since $P$ can be chosen arbitrarily, the affine extensions are not restricted to subgroups of $\Tn$ for any regular rooted $d$-ary tree $T$ themselves. With this setup, we discuss various results for the property $\p(H,G)$ for subgroups $H,G\leq \F^X\rtimes_\rho P$ that form an analogue to the properties defined for subgroups of $\Tn$.

\begin{remark} \label{Re: generalisation}
    We note that affine extensions, and all the related results, can further be generalised to \emph{wreath products} $\A\wr_\rho P:=\A^X\rtimes_\rho P$, where $\rho:P\to \text{Aut}(\A^X)$ is induced by the action of $P$ on $X$, and where $\A$ is any Abelian group instead of a field. However, since the construction of the counterexample to Conjecture \ref{Conj: Goksel} in Section \ref{Section: counterexample} relies heavily on the field properties of $\F$, we restrict our attention to the field case.
\end{remark}

In Section \ref{Section: Misc}, we prove that the local-global property $\p(H,G)$ can be reduced to a case where the groups $H$ and $G$ have the same order, and discuss various results with the same idea in mind. Next, we show that Theorem \ref{Thm: Goksel's proven properties}(b) is still true for affine extensions. In contrast to the proof of Theorem \ref{Thm: Goksel's proven properties}(b) that relies on the properties of 2-groups, our proof makes use of the structure of cocycles and coboundaries in the sense of group cohomology. 

In particular, in Section \ref{Section: conjugacy results} we show that, under our conditions, local conjugation of elements naturally corresponds to a \emph{local 1-coboundary}, and that global conjugation of elements naturally corresponds to a (global) \emph{1-coboundary}. In Section \ref{Section: group cohomology}, we show that these notions coincide for permutation representations, and therefore we are able to show the property $\p(H,G)$. Since the group cohomology does not depend on $\F$ being finite, the full result does not depend on it either, and is stated as follows.

\begin{thm} \label{Thm: IntroMain1}
    Let $\rho:P\to \aut(\F^X)$ be a permutation representation over a field $\F$, and let $G\leq \F^X\rtimes_\rho P$ be a subgroup of the corresponding affine extension such that $$G\cap (\F^X\times \{\one\})=\{(\zero,\one)\}.$$ Then, for any $H\leq \F^X\rtimes_\rho P$, $\p(H,G)$ holds.
\end{thm}

That is, Theorem \ref{Thm: IntroMain1} proves the properties $\p(H,G)$ when $N:=G\cap (\F^X\times \{\id\})$ is trivial. If $N$ is non-trivial, we can still try to apply this result, assuming only that the induced representation of $G/N$ is \emph{isomorphic} to a permutation representation, and is formally stated as follows.

\begin{thm} \label{Thm: IntroMain2}
        Let $\rho:P\to \aut(\F^X)$ be a permutation representation over a field $\F$, let $G\leq \F^X\rtimes_\rho P$ be a subgroup, let $N:=G\cap (\F^X\times \{\one\})$, and let $F\subseteq \F^X$ denote the subspace defined by $F\times \{\id\}=N$. Let $\rho_{G,N}:G/N\to \aut(\F^X/F)$ denote the induced representation as defined in (\ref{Rep: rho G N}).

        Assume that there exists a set $Y$ together with an isomorphism $\phi:\F^X/F\isomto \F^Y$, and assume that there exists a permutation representation $\tau:G/N\to \aut(\F^Y)$ such that the diagram \begin{equation} \label{Eq: isomorphism diagram}\begin{tikzcd}
            G/N \arrow[rd, "\tau"'] \arrow[r, "{\rho_{G,N}}"] & \aut(\mathbb{F}^X/F) \arrow[d, "\phi_*"] \\
             & \aut(\mathbb{F}^Y) 
            \end{tikzcd}\end{equation} commutes, where $\phi_*(A):=\phi\circ A\circ \phi^{-1}$. Then, for any $H\leq W$, $\p(H,G)$ holds.
    \end{thm}

When talking about the automorphism groups $\Tn=\F_2^{V_{n-1}}\rtimes \Tnn$ of the first $n$ levels of the binary rooted tree, the conditions of Theorem \ref{Thm: IntroMain2} are in particular satisfied for a subgroup $G\leq \Tn$ when there exists a direct sum $\F_2^{V_{n-1}}=F\oplus F'$ such that $G\cap (\F^X\times \{\id\}) = F\times \{\id\}$ and such that $G$ acts on $F'\leq \F^X$ by permuting some basis of $F'$. A specific demonstration of this idea is given in Example \ref{Ex: stabiliser}.

Finally, we stick to the automorphism groups $\Tn$ of the first $n$ levels of the binary rooted tree, and provide a detailed study of the structure of the representation $\rho$ when $H$ contains an element acting transitively on the $n$-th level of the tree, which we call an \emph{$n$-odometer}. The full result depends on a \emph{flag} $$\{\zero\}\subset E_1\subset E_2\subset\cdots \subset E_{2^{n-1}}=\F_2^{V_{n-1}}$$ of $\F_2$-subspaces, as constructed in Section \ref{Section: counterexample}. In particular, the global conjugacy of $H$ into another group $G$ is solely determined by the $n$-odometer.

\begin{thm} \label{IntroThm: counter}
    Let $n\geq 0$, let $T_n$ denote the first $n$ levels of the binary rooted tree, and let $H,G\leq \Tn$ be subgroups such that $H$ is elementwise $K_n$-conjugate into $G$. Assume that $H\cap G$ contains an $n$-odometer.
    
    Then, $H$ is globally $K_n$-conjugate into $G$ if and only if $H\subseteq G$.
\end{thm}

Theorem \ref{IntroThm: counter} is the basis for a counterexample of Conjecture \ref{Conj: Goksel}, showing that in order to either prove or disprove Conjecture \ref{Conj: GokselMarkov}, we would have to do more work. The methods used to get to the theorem, and therefore the counterexample, are described in full, as we hope that these techniques are valuable in general.

The paper is organised as follows. In Section \ref{Section: Preliminaries} we recall the basics of representation theory and introduce notation and definitions used throughout the paper. We continue to study some general properties of $\p(H,G)$ for permutation representations in Section \ref{Section: Misc}. After that, we introduce and prove some technical results needed from group cohomology in Section \ref{Section: group cohomology}, then use these results to prove the corresponding conjugacy results in Section \ref{Section: conjugacy results}. In Section \ref{Section: counterexample} we provide the counterexample to Conjecture \ref{Conj: Goksel}, as well as the methods used to get to this example. Finally, we end with a short discussion in Section \ref{Section: discussion}.

\section{Preliminaries} \label{Section: Preliminaries}

    In this section, we recall definitions and facts concerning the material used in the following chapters. In Sections \ref{Section: Preliminaries: Affine extensions} and \ref{Section: Preliminaries: Conjugation} we introduce the conjugacy problems using representation theory, and in Section \ref{Section: Preliminaries: rooted binary trees explanation} we consider this theory in the case of the binary rooted trees.

    \subsection{Affine extensions} \label{Section: Preliminaries: Affine extensions}

    The groups we study arise from representations of a finite group $P$ over a finite field $\F$. However, we stress that all of these definitions can be generalised to any group $P$ and field $\F$, and that Sections \ref{Section: group cohomology} and \ref{Section: conjugacy results} do not require $\F$ to be finite.
    
    \begin{defn}
        Let $P$ be a group. A \textbf{representation $\rho$ of $P$} is a group homomorphism $\rho:P\to \text{Aut}(F)$, where $F$ is a vector space over a field $\F$. For $h\in P,v\in F$, we generally write $h(v)$ instead of $\rho(h)(v)$ when there is no confusion.
    \end{defn}

    \begin{defn}
        Let $X$ be a set, and let $\F$ be a field. We define $$\F^X:=\{f:X\to \F\}$$ to be the $\F$-vector space of functions from $X$ to $\F$.
    \end{defn}

    These spaces have a standard basis given by indicator functions $\delta_x$.

    \begin{defn} \label{Def: indicator function}
        Let $X$ be a set, and let $\F$ be a field. For each $x\in X$, we define the \textbf{indicator functions} $\delta_x:X\to \F$ by $$\delta_x(y):=\left\{\begin{matrix}
            1, & \text{if }y=x;\\
            0, & \text{if }y\neq x.
        \end{matrix}\right.$$
    \end{defn}

    Suppose that we have a group $P$, together with an action of $P$ on a set $X$. Then for each $g\in P$, the assignment \begin{align}\label{The functions A_g}
        A_g:\F^X\to \F^X, \phantom{;;;}t\mapsto t\circ g^{-1},
    \end{align} where on the right we mean to compose the maps $g^{-1}:X\to X$ and $t:X\to \F$, is a linear map. Indeed, $A_g$ acts on $\F^X$ by permuting the basis functions $\delta_x$ according to the action of $g^{-1}$ on $X$. More specifically, if $y\in X$, then \begin{align} \label{Eq: A_g acting on delta_x}
        A_g(\delta_x)(y)=\delta_x(g^{-1}(y))=\left\{\begin{matrix}
        1, &\text{if }p^{-1}(y)=x;\\
        0, &\text{if }p^{-1}(y)\neq x.
    \end{matrix}\right.
    \end{align} And therefore we get $A_g(\delta_x)=\delta_{g^{-1}(x)}$. Note that for any $g,h\in P$ we have the relation \begin{align}\label{Eq: linear maps commute with group}
        A_gA_h = A_{gh},
    \end{align} and that the matrices $A_g$ (with respect to the basis $\{\delta_x\}_{x\in X}$) are permutation matrices. 

    \begin{defn} \label{Def: permutation representation}
        Let $P$ be a finite group, let $X$ be a finite set, and let $\F$ be a finite field. A representation $\rho:P\to \text{Aut}\left(\F^X\right)$ is called a \textbf{permutation representation} if there exists an action of $P$ on $X$ such that for any $g\in P$: $$\rho(g)=A_g,$$ where $A_g:\F^X\to \F^X$ is the linear map as described in (\ref{The functions A_g}).
    \end{defn}

    In the paper \cite{Goksel}, the main point of interest are the automorphism groups of the first $n$ levels of a rooted (regular) binary tree $T_n$, denoted as $\Tn$. These groups follow a recursive identification $$\Tn\cong \F_2^{V_{n-1}}\rtimes \Tnn,$$ which we describe more rigorously in Section \ref{Section: Preliminaries: rooted binary trees explanation}. This notation, as well as the methods used in \cite{Goksel}, inspires the use of the notation $\F^X\rtimes P$, where $X$ is a finite set, $P$ is a finite group, and $\F$ is a finite field.
    
    \begin{defn}
        Let $\rho:P\to \text{Aut}(\F^X)$ be a permutation representation. The \textbf{affine extension of $\rho$} is the semidirect product $\F^X\rtimes_\rho P$. That is, the elements of $\F^X\rtimes_\rho P$ are given as pairs $(t,g)\in \F^X\times P$, and the group operation is given by the rule \begin{align}
            \label{Eq: permutation representation}(s,h)\cdot (t,g):=(s+\rho(h)(t),hg)=(s+A_h(t),hg).
        \end{align}
    \end{defn}
    
    We denote the identity element of $\F^X\rtimes_\rho P$ by $(\zero ,\one)$, where $\zero:X\to \F$ is the zero function, and $\one\in P$ is the trivial element.
    
    Finally, recall the following technical fact for normal subgroups. We omit the (standard) proof.
    
    \begin{lem} \label{Lem: normal form}
        Let $G$ be a group, let $N\trianglelefteq G$ be a normal subgroup, and assume that $H\leq G$ is a subgroup such that $G=NH$. Then, for any $g\in G$, we may write $g=nh$ for some $n\in N$ and $h\in H$. 
    \end{lem}

    \subsection{Conjugation within affine extensions} \label{Section: Preliminaries: Conjugation} In this subsection, we work out useful conjugation relations for permutation representations. Define $$\pi:\F^X\rtimes_\rho P\to P$$ to be the projection map. Clearly, the kernel of $\pi$ equals $K:=\F^X\times \{\one\}$.

    Each element $g\in \F^X\rtimes_\rho P$ also has a natural action on $X$ via the map $\pi$. Hence, we naturally define the linear maps $A_g$ for $g\in \F^X\rtimes_\rho P$ by setting $$A_g:=A_{\pi(g)},$$ where $A_{\pi(g)}$ is defined as in (\ref{The functions A_g}). Our motivation for this notation lies with conjugation, since for $g=(s,h)\in \F^X\rtimes_\rho P$ and $(t,\one)\in K=\ker(\pi)$ we have\begin{equation} \label{Eq: conjugating t by some g}
        \begin{aligned}
            g(t,\one)g^{-1}&=(s,h)(t,\one)(s,h)^{-1}=(s+A_h(t),h)(s,h)^{-1}\\  &=(A_h(t),\one)(s,h)(s,h)^{-1}=(A_h(t),\one)=(A_g(t),\one).
        \end{aligned}
    \end{equation} Note that in (\ref{Eq: conjugating t by some g}), if $g=(s,\one)\in \F^X\times \{\one\}$, then $(t,\one)$ remains fixed. On the other hand, if we interchange the elements $g=(s,h)$ and $(t,\one)$ as in (\ref{Eq: conjugating t by some g}), then \begin{align}\label{Eq: conjugating g by some t}
        (t,\one)g(t,\one)^{-1}=(t,\one)(s,h)(t,\one)^{-1}=(t-A_h(t),\one)(s,h)=(t-A_h(t),\one)g
    \end{align} where we use that $(t,\one)^{-1}=(-t,\one)$.

    The next definition is an analogue of Definitions \ref{IntroDef: defn7.2} and \ref{IntroDef: defn7.3}, where $\pi : \F^X\rtimes_\rho P\to P$ plays the role of $\pi_n:W_n\to W_{n-1}$ and $K:=\ker \pi$ plays the role of $K_n:=\ker \pi_n$.

    \begin{defn} \label{Def: K-conjugacy}
        Let $H,G\leq \F^X\rtimes_\rho P$. We say that $H$ is \textbf{elementwise $K$-conjugate into $G$} if for every $h\in H$ there exists some $f\in K$ such that $fhf^{-1}\in G$. We say that $H$ is \textbf{globally $K$-conjugate into $G$} if there exists some $f\in K$ such that $fHf^{-1}\subseteq G$. Finally, we say that $\p_K(H,G)$ \textbf{holds} if:
    
        \begin{center}$H$ is elementwise $K$-conjugate into $G$ if and only if $H$ is globally $K$-conjugate into $G$.\end{center}
    
        Generally, when $K$ is understood, we write $\p(H,G)$ instead of $\p_K(H,G)$.
    \end{defn}

    \begin{lem}\label{Lem: Elementwise implies |H| <= |G|}
        Suppose that $H$ is elementwise $K$-conjugate into $G$. Then $|H|\leq|G|$.
    \end{lem}
    
    \begin{proof}
        If $t\in \F^X$ and $g\in \F^X\rtimes_\rho P$, then by (\ref{Eq: conjugating g by some t}) we have $\pi((t,\one)g(t,\one)^{-1})=\pi(g)$. Moreover, by (\ref{Eq: conjugating t by some g}), we know that each element of $H\cap K$ is fixed by conjugation by any element from $K$. As a result, $H\cap K\subseteq G\cap K$. Using the first isomorphism theorem, we find: $$|H|=|H\cap K|\cdot |\pi(H)|\leq |G\cap K|\cdot |\pi(G)|=|G|.$$
    \end{proof}

    Within the upcoming sections, we consider linear maps given by sums and products of the maps $A_g$. For any subgroup $G\subseteq \F^X\rtimes_\rho P$, define \begin{center}
        \fbox{$R(G):=\F\langle A_g\mid g\in G\rangle\subseteq \text{End}_\F\left(\F^X\right)$}
    \end{center} to be the $\F$-subalgebra of $\text{End}_\F\left(\F^X\right)$ generated by the linear maps $\{A_g\mid g\in G\}$. We write $I:=A_{(\zero,\one)}$ for the linear map corresponding to the trivial element $(\zero,\one)\in G$, and note that $I:\F^X\to \F^X$ is the identity map. For the rest of the paper, whenever we add or multiply linear maps $A_g$, we mean the corresponding linear map as an element of $R(G)$. 
    
    Note that multiplication and addition in $R(G)$ naturally correspond to elementwise conjugation of $\F^X\times \{\one\}$ by $g,h\in G$. Namely, if $t\in \F^X$, we have: $$(A_gA_h(t),\one)=gh(t,\one)(gh)^{-1}=(A_{gh}(t),\one),$$ and $$((A_g+A_h)(t),\one)=(A_g(t),\one)(A_h(t),\one)=g(t,\one)g^{-1}h(t,\one)h^{-1},$$ where we use (\ref{Eq: conjugating t by some g}) in both cases.
        
    \subsection{Rooted binary trees} \label{Section: Preliminaries: rooted binary trees explanation}
        Examples of affine extensions, as studied in \cite{Goksel}, include automorphism groups of rooted binary trees. In this subsection, we formally introduce related definitions, as well as show that these groups are isomorphic to affine extensions by $\F_2$-vector spaces.

        We define the infinite rooted binary tree $T=(V,E)$, together with a natural labelling, as the following graph. Consider the alphabet $X=\{0,1\}$, and let $\varepsilon$ denote the empty word over $X$. We set $V_0:=\{\varepsilon\}$, and for each $k\geq 1$, we define $$V_k:=\{v_1v_2\cdots v_k\mid v_i\in X\}$$ to be the set of words of length $k$ over $X$. The vertices of $T$ are given by $$V:=\bigsqcup_{k\geq 0} V_k,$$ and we connect vertices $v\in V_k$ and $w\in V_{k+1}$ with an edge if and only if $w=vx$ for some $x\in X$. The sets $V_k$ denote the \emph{$k$-th level sets}. For any $n\geq 0$, we let $T_n$ denote the subgraph of $T$ with vertex set $\bigsqcup_{k=0}^n V_k$, and as edge set all edges of $T$ that connect vertices of $T_n$. See Figure \ref{fig:T_3 for binary} for $T_3$.

        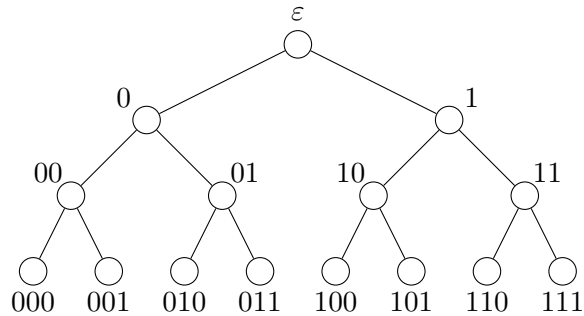
\begin{figure}[h!] 
        \begin{center}
            \begin{tikzpicture}[main/.style = {draw, circle}]

                \node[main] (0) at (0,0) {};
            
                \node[main] (11) at (-2,-1) {};
                \node[main] (12) at (2,-1) {};
            
                \node[main] (21) at (-3,-2) {};
                \node[main] (22) at (-1,-2) {};
                \node[main] (23) at (1,-2) {};
                \node[main] (24) at (3,-2) {};
            
                \node[main] (1) at (-3.5,-3) {};
                \node[main] (2) at (-2.5,-3) {};
                \node[main] (3) at (-1.5,-3) {};
                \node[main] (4) at (-0.5,-3) {};
                \node[main] (5) at (0.5,-3) {};
                \node[main] (6) at (1.5,-3) {};
                \node[main] (7) at (2.5,-3) {};
                \node[main] (8) at (3.5,-3) {};
            
                \node at (0,.4) {$\varepsilon$};

                \node at (-2.3,-.7) {$0$};
                \node at (2.3,-.7) {$1$};

                \node at (-3.3,-1.7) {$00$};
                \node at (-.7,-1.7) {$01$};
                \node at (.7,-1.7) {$10$};
                \node at (3.3,-1.7) {$11$};

                \node at (-3.5,-3.4) {$000$};
                \node at (-2.5,-3.4) {$001$};
                \node at (-1.5,-3.4) {$010$};
                \node at (-0.5,-3.4) {$011$};
                \node at (0.5,-3.4) {$100$};
                \node at (1.5,-3.4) {$101$};
                \node at (2.5,-3.4) {$110$};
                \node at (3.5,-3.4) {$111$};
            
                \draw (0)--(11);
                \draw (0)--(12);
                \draw (11)--(21);
                \draw (11)--(22);
                \draw (12)--(23);
                \draw (12)--(24);
                \draw (21)--(1);
                \draw (21)--(2);
                \draw (22)--(3);
                \draw (22)--(4);
                \draw (23)--(5);
                \draw (23)--(6);
                \draw (24)--(7);
                \draw (24)--(8);
            
            \end{tikzpicture}
        \end{center}
        \caption{The first three levels of a regular rooted binary tree.}
        \label{fig:T_3 for binary}
    \end{figure}

        In addition, for some vertex $v\in V_k$ we define $vT$ to be the subgraph of $T$ with vertex set consisting of the words starting with $v$, and with edge set all edges of $T$ that connect vertices of $vT$. By identifying $v$ as the new root of $vT$ we can readily see that the induced map \begin{align}\label{Eq: section map}\eta_v:vT\to T, \phantom{::::}vw\mapsto w,\end{align} is a graph isomorphism.

        Fix an integer $n\geq 1$, consider the graph $T_n$, and define $\Tn$ as its automorphism group. That is, it is the set of bijections $\bigsqcup_{k=0}^nV_k\to \bigsqcup_{k=0}^nV_k$ that preserve adjacency of vertices in $T_n$. Automorphisms of graphs preserve degrees of vertices, and therefore, since $\varepsilon$ is the unique vertex of degree 2, the root $\varepsilon$ is always fixed.

        In addition, since $V_k$ is the set of vertices such that the shortest path to $\varepsilon$ consists of exactly $k$ edges, each $g\in \Tn$ provides a bijection of $V_k$ for $0\leq k\leq n$. With this in mind, we define \emph{sections} as follows.
        
        \begin{defn}
            Let $g\in \Tn$, $0\leq k\leq n$ be an integer, and let $v\in V_k$. The \textbf{section} of $g$ at $v$, denoted $g_v$, is given by $$g_v:=\eta_{g(v)}\circ g\circ \eta_v^{-1}\in \aut(T_{n-k}),$$ where $\eta_v:vT\to T$ is the graph isomorphism defined in (\ref{Eq: section map}).
        \end{defn}

        As noted before, each $g\in\Tn$ can be restricted to $T_k$ for any $0\leq k\leq n$. That is, we have well defined (surjective) homomorphisms $\Tn\to \aut(T_k)$ given by $g\mapsto g|_{T_k}$. When $k=n-1$, we particularly let \begin{align*}
            \pi_n:\Tn\to \Tnn,\phantom{;;;}g\mapsto g|_{T_{n-1}}
        \end{align*} and define $K_n:=\ker(\pi_n)$.

        Next, we demonstrate how to obtain $\Tn$ from $\Tnn$ recursively, in two different manners. The first, which we return to in Chapter \ref{Section: counterexample}, is given as follows.

        Let $g\in \Tn$, let $v\in V_k$ for $1\leq k\leq n$, and write $v=xw$ for $w\in V_{k-1}$ and $x\in \{0,1\}$. Since the sections $g_0,g_1$ of $g$ are automorphisms of $T_{n-1}$, we must get \begin{align} \label{Eq: formula acting on words}
            g(v)=g(xw)=\sigma_g(x)g_{\sigma_g^{-1}(x)}(w),
        \end{align} where $\sigma_g\in \aut(T_1)$ is a bijection of $V_1=\{0,1\}$, and where $g_0,g_1\in \Tnn$. It follows that we may write \begin{align} \label{Eq: original wreath recursion}
            g=(g_{\sigma_g^{-1}(0)},g_{\sigma_g^{-1}(1)})\sigma_g.
        \end{align} That is, $g$ acts on the word $xw\in V_k$ by first acting with $\sigma_g$ on $x$, and then by applying $g_{\sigma^{-1}_g(x)}$ to the subtree $xT$. In particular, from (\ref{Eq: formula acting on words}) it can readily be seen that \begin{align} \label{Eq: Forward wreath recursion}
            \Tn\cong \Tnn^{V_1}\rtimes \text{Aut}(T_1)\cong \Tnn^2\rtimes S_2,
        \end{align} where $\Tnn^{V_1}$ is the set of functions $f:V_1\to \Tnn$, forming a group with coordinate-wise multiplication, and where the action of $\sigma\in\aut(T_1)$ on $\Tnn^{V_1}$ is given by $\sigma(t)=t\circ \sigma^{-1}:V_1\to \Tnn$. That is, $\sigma\in \aut(T_1)$ permutes the $V_1$ copies of $\Tnn$ according to the action of $\sigma^{-1}$ on $V_1$.

        \begin{example} \label{Ex: the element (01)}
            Consider the automorphism $g\in\Tn$ that acts on a word $v=v_1\cdots v_n\in V_n$ by $$g(v_1\cdots v_n)=(1-v_1)v_2\cdots v_n.$$ That is, it only changes the first letter $v_1\in \{0,1\}$. Let $(0\ 1)\in S_2$ denote the non-trivial permutation of $\{0,1\}$. Since $g_v=\one$ whenever $v\neq \varepsilon$, it follows that we can write $$(0\ 1):=g=(\one,\one)\cdot (0\ 1)\in \Tn^2\rtimes S_2,$$ using the notation of (\ref{Eq: original wreath recursion}) and (\ref{Eq: Forward wreath recursion}).
        \end{example}

        \begin{example} \label{Ex: the element (g_0,g_1)}
            Consider any pair of automorphisms $g_0,g_1\in \Tnn$, and let $g$ denote the automorphism such that for any word $v\in V_{n-1}$ we have $$g(iv)=ig_i(v),$$ for $i=0,1$. Then $g$ acts trivally on words of length 1, and therefore, we can write $$(g_0,g_1):=g=(g_0,g_1)\cdot \one \in \Tnn^2\rtimes S_2,$$ following the notation of (\ref{Eq: original wreath recursion}) and (\ref{Eq: Forward wreath recursion}).
        \end{example}

        Instead of describing the recursion for words $v=xw$ with $v\in V_k$, $x\in \{0,1\}$ and $w\in V_{k-1}$, we can alternatively describe the recursion for the words of the form $wx$. In a similar fashion to (\ref{Eq: formula acting on words}), taking $v=wx$, we get \begin{align}\label{Eq: formula acting on words from the back}g(v)=g(wx)=g(w)g_{g^{-1}(w)}(x).\end{align} It follows that (\ref{Eq: formula acting on words from the back}) gives an isomorphism \begin{align}\label{Eq: backward wreath recursion}
            \Tn\cong \aut(T_1)^{V_{n-1}}\rtimes \Tnn\cong S_2^{V_{n-1}}\rtimes \Tnn,
        \end{align} where the action of $g\in \Tnn$ on $\aut(T_1)^{V_{n-1}}$ is given by $g(t)=t\circ g^{-1}:V_{n-1}\to \aut(T_1)$. When identifying $\aut(T_1)\cong \F_2$, we can further view $\aut(T_1)^{V_{n-1}}$ as a vector space over $\F_2$. In particular, each $g\in \Tnn$ acts on $\F_2^{V_{n-1}}$ by permuting the standard basis elements, so that $$\Tn\cong \F_2^{V_{n-1}}\rtimes \Tnn$$ is the affine extension of a permutation representation $\rho:\Tnn\to \F^{V_{n-1}}_2$.

        To be precise, for each $g\in \Tnn$, the maps $A_g$ as in (\ref{The functions A_g}) are given by \begin{align}
            A_g:\F_2^{V_{n-1}}\to \F_2^{V_{n-1}}\phantom{;;;}t\mapsto t\circ (g^{-1}|_{V_{n-1}}),
        \end{align} and the induced representation $$\rho:\Tnn\to \aut\left(\F_2^{V_{n-1}}\right),\phantom{;;;} g\mapsto A_g$$ is a permutation representation by Definition \ref{Def: permutation representation}. Hence, we can apply the material from Sections \ref{Section: Preliminaries: Affine extensions} and \ref{Section: Preliminaries: Conjugation}. Just like in Section \ref{Section: Preliminaries: Conjugation}, we define $A_g$ for $g\in \Tn$, via $$A_g:=A_{\pi(g)}=A_{g|_{T_{n-1}}},$$ obtaining a representation of $\F_2^{V_{n-1}}\rtimes \Tnn=\Tn$ onto $\aut(\F^{V_{n-1}}_2)$. The linear maps $A_g:\F_2^{V_{n-1}}\to \F_2^{V_{n-1}}$ can be computed recursively using the semi-direct product structure $g\in \Tnn^2\rtimes S_2$ as in (\ref{Eq: original wreath recursion}). The following result can be used as a blueprint for this action, making use of the fact that, if $t\in \F_2^{V_{n-1}}$, we can naturally split $t$ into two maps $t|_0,t|_1:V_{n-2}\to \F_2$ by setting for $i\in \{0,1\}$: $$t|_i:V_{n-2}\to \F_2,\phantom{;;;} v\mapsto t(iv),$$ where $iv\in V_{n-1}$ is the concatenation of $i$ and $v\in V_{n-2}$. We leave out the proof, since it is a straightforward computation.

        \begin{prop} \label{Prop: blueprint for action}
            Split $\F_2^{V_{n-1}}\cong\F_2^{V_{n-2}}\oplus \F_2^{V_{n-2}}$ via the isomorphism $$\Phi:t\mapsto (v\mapsto t(0v),v\mapsto t(1v)).$$ For any $g\in \Tn$, define the action of $A_g$ on $\F_2^{V_{n-2}}\oplus \F_2^{V_{n-2}}$ by $$A_g(t_0,t_1):=\Phi(A_g(\Phi^{-1}(t_0,t_1))).$$ Then, for $(0\ 1):=(\one,\one)\cdot (0\ 1)\in \Tnn^2\rtimes S_2=\Tn$, the induced action of $A_{(0\ 1)}$ under $\Phi$ is given by $$A_{(0\ 1)}(t_0,t_1)=(t_1,t_0).$$ Furthermore, for $(g_0,g_1):=(g_0,g_1)\cdot \one\in \Tnn^2\rtimes S_2=\Tn$, the induced action of $A_{(g_0,g_1)}$ under $\Phi$ is given by $$A_{(g_0,g_1)}(t_0,t_1)=(A_{g_0}(t_0),A_{g_1}(t_1)).$$
        \end{prop}

        Since the elements $(0\ 1)$ and $(g_0,g_1)$ (where $g_0,g_1\in \Tnn$) as in Examples \ref{Ex: the element (01)} and \ref{Ex: the element (g_0,g_1)} generate $\Tn$, the combination of Proposition \ref{Prop: blueprint for action} and the relation (\ref{Eq: linear maps commute with group}) makes it possible to compute any linear map $A_g=A_{(g_{\sigma^{-1}(0)},g_{\sigma^{-1}(1)})}A_{\sigma}$ whenever $g=(g_{\sigma^{-1}(0)},g_{\sigma^{-1}(1)})\sigma$.

    \subsection{Notation}\label{Section: notation} Throughout the rest of the paper, we use the following notation:

    \begin{itemize}
        \item We let $P$ be a finite group.
        \item We let $\F$ be a finite field, with the exception of Sections \ref{Section: group cohomology} and \ref{Section: conjugacy results}, where $\F$ may be infinite.
        \item We let $X$ be a finite set, together with an action of $P$.
        \item For each $g\in P$, we let $A_g:\F^X\to \F^X$ denote the linear map $t\mapsto t\circ g^{-1}$.
        \item We let $\rho: P\to \aut(\F^X)$ given by $\rho(g):=A_g$ denote the associated permutation representation.
        \item We let $\F^X\rtimes_\rho P$ denote the corresponding affine extension of $P$.
        \item We let $(\zero,\one)\in \F^X\rtimes_\rho P$ denote the identity element.
        \item We let $\pi:\F^X\rtimes_\rho P\to P$ denote the restriction homomorphism.
        \item We let $K:=\ker(\pi)=\F^X\times\{\id\}$.
    \end{itemize}

    As we have noted in Remark \ref{Re: generalisation}, each of these notions, and all results outside of Section \ref{Section: counterexample}, can be generalised to wreath products $\A\wr P$ by taking an arbitrary Abelian group $\mathbb{A}$ instead of the field $\F$, since the field properties are not used throughout the proofs.

\section{Reducing local-global conjugacy to the equal order case} \label{Section: Misc}

    In this section, we discuss the properties $\p(H,G)$ for subgroups $H,G\leq \F^X\rtimes_\rho P$ as described in Section \ref{Section: Preliminaries}. Since determining whether a property $\p(H,G)$ holds or not can be difficult, it can be useful to know when we can replace this property by an equivalent property for another pair of subgroups. Throughout this section, we use the notation of Section \ref{Section: notation}. The following result holds.

    \begin{thm} \label{thm: reducing the problem to equal size}
        Let $H,G\leq \F^X\rtimes_\rho P$. Define $G_0:=\pi^{-1}(\pi(H))\cap G$ and $H_0:=H\cdot (G\cap K)$. Then the following are equivalent:
        
        \begin{itemize}
            \item[(i)] $\p(H,G)$;
            \item[(ii)] $\p(H,G_0)$;
            \item[(iii)] $\p(H_0,G)$;
            \item[(iv)] $\p(H_0,G_0)$.
        \end{itemize}
        
        Moreover, if $H$ is elementwise $K$-conjugate into $G$, then $|H_0|=|G_0|$.
    \end{thm}

    Since for any $h\in \F^X\rtimes_\rho P$ and $v\in K$ we have $\pi(h)=\pi(vhv^{-1})$, the following claim is readily true, and is used several times throughout the proof.

    \begin{claim} \label{Claim: equivalence}
        If for some $h\in H_0$ and $v\in K$ we have $vhv^{-1}\in G$, then we also have $vhv^{-1}\in G_0:=\pi^{-1}(\pi(H))\cap G$.
    \end{claim}
    
    \begin{proof}[Proof of Theorem \ref{thm: reducing the problem to equal size}]
        (i)$\Rightarrow$(ii). Suppose that $H$ is elementwise $K$-conjugate into $G_0$. Since $G_0$ is a subgroup of $G$, we know that $H$ is elementwise $K$-conjugate into $G$. By $\p(H,G)$ we know that $H$ is globally $K$-conjugate into $G$. By Claim \ref{Claim: equivalence}, $H$ is also globally $K$-conjugate into $G_0$. That is, $\p(H,G_0)$ holds.
    
        (ii)$\Rightarrow$(iii). Suppose that $H_0$ is elementwise $K$-conjugate into $G$. Then also $H\subseteq H_0$ is elementwise $K$-conjugate into $G$. By Claim \ref{Claim: equivalence} we see that $H$ is elementwise $K$-conjugate into $G_0$. Hence by $\p(H,G_0)$, we know that $H$ is globally $K$-conjugate into $G_0$, and therefore $H$ is globally $K$-conjugate into $G$, say $vHv^{-1}\subseteq G$. But since $G\cap K$ is pointwise fixed under conjugation by $K$ by (\ref{Eq: conjugating t by some g}), we get $$vH_0v^{-1}=vHv^{-1}\cdot v(G\cap K)v^{-1}\subseteq G\cdot (G\cap K)=G$$ Hence $\p(H_0,G)$ holds.
    
        (iii)$\Rightarrow$(iv). $\p(H_0,G)\Rightarrow \p(H_0,G_0)$ has the same proof as $\p(H,G)\Rightarrow \p(H,G_0)$, by substituting $H_0$ for $H$.
    
        (iv)$\Rightarrow$(i) Suppose $H$ is elementwise $K$-conjugate into $G$. We first show that $G\cap K$ is a normal subgroup of $H_0$. Indeed, $G\cap K$ is pointwise fixed under conjugation by $G\cap K$ by (\ref{Eq: conjugating t by some g}), and for any $h=(t,f)\in H$ we can choose $g\in G$ such that $g=(s,f)$ by (\ref{Eq: conjugating g by some t}). As a result, we get \begin{align} \label{Eq: conjugation doesn't change}
            h(t,\one)h^{-1}=(A_h(t),\one)=(A_g(t),\one)=g(t,\one)g^{-1}
        \end{align} by (\ref{Eq: conjugating t by some g}). Since $G\cap K$ is a normal subgroup of $G$, (\ref{Eq: conjugation doesn't change}) implies that $G\cap K$ is a normal subgroup of $H\cdot (G\cap K) =H_0$.

        By Lemma \ref{Lem: normal form}, we can write each element of $H_0$ as $hf$ for some $h\in H$ and $f\in G\cap K$. Now choose $v\in K$ such that $vhv^{-1}\in G$. Then $$v(hf)v^{-1}=vhv^{-1}\cdot vfv^{-1}=vhv^{-1}\cdot f\in G\cdot (G\cap K)=G,$$ so $H_0$ is elementwise $K$-conjugate into $G$. By Claim \ref{Claim: equivalence} it follows that $H_0$ is elementwise $K$-conjugate into $G_0$. By $\p(H_0,G_0)$, we know that $H_0$ is globally $K$-conjugate into $G_0$, say $vH_0v^{-1}\subseteq G_0$ for some $v\in K$. We conclude that $$vHv^{-1}\subseteq vH_0v^{-1}\subseteq G_0\subseteq G,$$ showing $\p(H,G)$.
        
        Finally, we show that $|H_0|=|G_0|$ whenever $H$ is elementwise $K$-conjugate into $G$. As shown in the proof of (iv)$\Rightarrow$(i), this implies that $H_0$ is elementwise $K$-conjugate into $G_0$. Thus, since $H_0\cap K$ is pointwise fixed under conjugation by $K$ by (\ref{Eq: conjugating t by some g}), we get $H_0\cap K\subseteq G_0\cap K$. But by definition we have $$G_0\cap K=\left(\pi^{-1}\pi(H)\cap G\right)\cap K\subseteq G\cap K\subseteq H_0,$$ showing that $H_0\cap K=G_0\cap K=G\cap K$.
    
        Similarly, using that $\pi(vh)=\pi(hv)=\pi(h)$ for any $h\in \F^X\rtimes_\rho P$ and $v\in K$, we have $$\pi(H_0)=\pi(H)=\pi(\pi^{-1}\pi(H)\cap G)=\pi(G_0),$$ where at the second equality we used that $\pi(H)\subseteq \pi(G)$, which follows from the assumption that $H$ is elementwise $K$-conjugate into $G$. By the first isomorphism theorem: $$|H_0|=|H_0\cap K|\cdot |\pi(H_0)|=|G\cap K|\cdot |\pi(H_0)|= |G_0\cap K|\cdot |\pi(G_0)|=|G_0|,$$ finishing the proof.
    \end{proof}

    Theorem \ref{thm: reducing the problem to equal size} implies that, working with $\p(H,G)$, we can restrict to groups of equal size. That is, we have the following corollary.
    
    \begin{cor}
        Suppose that $\p(H,G)$ holds for all $H,G\leq \F^X\rtimes_\rho P$ with $|H|=|G|$. Then $\p(H,G)$ holds for all $H,G\leq \F^X\rtimes_\rho P$.
    \end{cor} \vspace{-.6cm} \qed
    
    Next, when we have $|H|=|G|$, we find that the property $\p(H,G)$ is symmetric.
    
    \begin{prop}
        Let $H,G\leq \F^X\rtimes_\rho P$ be subgroups such that $|H|=|G|$. Then $\p(H,G)$ holds if and only if $\p(G,H)$ holds.
    \end{prop}
    
    \begin{proof}
        We show that if $H$ is elementwise (resp. globally) $K$-conjugate into $G$, then conversely $G$ is elementwise (resp. globally) $K$-conjugate into $H$.
    
        First suppose $H$ is elementwise $K$-conjugate into $G$. Since $\pi(vhv^{-1})=\pi(h)$ for any $h\in \F^X\rtimes_\rho P$ and $v\in K$, we get $\pi(H)\subseteq \pi(G)$. Moreover, if $N_H:=H\cap K$ and $N_G:=G\cap K$, then $N_H$ is elementwise conjugate into $N_G$, which shows $N_H\subseteq N_G$ since $K$ is Abelian. The first isomorphism theorem tells us that $$|H|=|N_H|\cdot |\pi(H)|\leq |N_G|\cdot |\pi(G)|=|G|.$$ Since $|G|=|H|$ (and $N_H\subseteq N_G$ and $\pi(H)\subseteq \pi(G)$), we get that $\pi(H)=\pi(G)$ and $N_H=N_G$. 
        
        Let $g\in G$ and choose any $h\in H$ with $\pi(h)=\pi(g)$, which we can do since $\pi(H)=\pi(G)$. Then let $v\in K$ such that $vhv^{-1}\in G$. Now, $\pi(g)=\pi(vhv^{-1})$, so $g^{-1}vhv^{-1}\in N_G$, say that it equals $n\in N_G=N_H$. It follows that $vhv^{-1}=gn$, which we rewrite to $$v^{-1}gv=hn^{-1}\in HN_G=H,$$ making use of $vnv^{-1}=n$ by (\ref{Eq: conjugating g by some t}). Since $g\in G$ was arbitrary, we get that $G$ is elementwise $K$-conjugate into $H$.
    
        If instead we know that $H$ is globally $K$-conjugate into $G$, say $vHv^{-1}\subseteq G$, then $vHv^{-1}=G$ follows by a counting argument. As a result, $G^{-v}=H$, so that $G$ is globally $K$-conjugate into $H$.
    \end{proof}

    In Theorem \ref{thm: reducing the problem to equal size}, we showed that $\p(H,G)$ is equivalent to $\p(H_0,G)$, where $H_0=H\cdot (G\cap K)$. Instead of taking $G\cap K$ here, we can also try to take other subgroups $A\leq K$. And indeed, the following lemma holds.
    
    \begin{lem} \label{Lem: p(H,G) implies p(AH,G)}
        Suppose $\p(H,G)$ holds. Then, for any subgroup $A\leq K$, $\p(AH,G)$ holds. 
    \end{lem}
    
    \begin{proof}
        If $AH$ is elementwise $K$-conjugate into $G$, then $H$ is elementwise $K$-conjugate into $G$. By $\p(H,G)$ we know that $vHv^{-1}\subseteq G$ for some $v\in K$. As $K$ is Abelian, $vAv^{-1}=A$, and hence $v(AH)v^{-1}=vAv^{-1}vHv^{-1}\subseteq AG$. Moreover, since $A$ is elementwise $K$-conjugate into $G$, we know that $A\subseteq G$. We conclude that $v(AH)v^{-1}\subseteq G$, so that $\p(AH,G)$ holds.
    \end{proof}
    
    Note that $A$ doesn't have to be a subgroup of $G$. A natural follow-up question would be to ask whether it is still possible if $A$ is not a subgroup of $K$. To do so, we consider the following definition.
    
    \begin{defn}
        Let $A\subseteq \F^X\rtimes_\rho P$ be a subgroup. We define: $$D(A):=\{v\in K\mid vAv^{-1}= A\}=K\cap N_{\F^X\rtimes_\rho P}(A),$$ where $N_{\F^X\rtimes_\rho P}(A)$ is the normaliser of $A$ in $\F^X\rtimes_\rho P$.
    \end{defn}

    Then, if $D(A)=K$, and $A$ is a normal subgroup of $G$, we can indeed obtain a result similar to Lemma \ref{Lem: p(H,G) implies p(AH,G)}
    
    \begin{lem} \label{Lem: D(F) = K}
        Consider groups $G,H\leq \F^X\rtimes_\rho P$, and let $A\trianglelefteq G$ be a normal subgroup such that $D(A)=K$. Then $\p(AH,G)$ holds if and only if $\p(H,G)$ holds.
    \end{lem}

    Before we prove this, we show the following claim.

    \begin{claim} \label{Claim: D(F) = K}
            If $H$ is elementwise $K$-conjugate into $G$, then $A$ is a normal subgroup of $AH$.
        \end{claim}
    
        \begin{proof}
            Let $h\in H$, and choose $v\in K$ such that $g:=vhv^{-1}\in G$. Since $A$ is a normal subgroup of $G$, we get $$hAh^{-1}=v^{-1}gvAv^{-1}g^{-1}v=v^{-1}gAg^{-1}v=v^{-1}Av=A,$$ where we use that $D(A)=K$. Therefore $A$ is invariant under conjugation by $h\in H$, and it follows that $A$ is a normal subgroup of $AH$. 
        \end{proof}
    
    \begin{proof}[Proof of Lemma \ref{Lem: D(F) = K}]
        Suppose that $\p(AH,G)$ holds, and assume that $H$ is elementwise $K$-conjugate into $G$. By Claim \ref{Claim: D(F) = K} and Lemma \ref{Lem: normal form}, each element of $AH$ is of the form $ah$ for $a\in A$ and $h\in H$. Therefore, if $v\in K$ is such that $vhv^{-1}\in G$, we get $$v(ah)v^{-1}=vav^{-1}\cdot vhv^{-1}\in AG=G.$$ That is, $AH$ is elementwise $K$-conjugate into $G$. By $\p(AH,G)$, $AH$ is globally $K$-conjugate into $G$, so that $H\subseteq AH$ is also globally $K$-conjugate into $G$. We conclude that $\p(H,G)$ holds.
    
        Suppose that $\p(H,G)$ holds, and assume that $AH$ is elementwise $K$-conjugate into $G$. Then $H$ is elementwise $K$-conjugate into $G$, so by $\p(H,G)$ we know that there exists some $v\in K$ such that $vHv^{-1}\subseteq G$. It follows that $$v(AH)v^{-1}=vAv^{-1}\cdot vHv^{-1}\subseteq AG=G.$$
    \end{proof}

\section{Group cohomology} \label{Section: group cohomology}

    In this section, we recall a couple of results and definitions from group cohomology. For more details about group cohomology, we refer to \cite{Serre}. We only use the first cohomology group with coefficients in the Abelian group $\F^X$, as defined below. Contrary to the notation from Section \ref{Section: notation}, we note that the field $\F$ does not need to be finite for the definitions and results in this section to be true, as its finiteness is never used. Let $G\leq \F^X\rtimes_\rho P$ be a subgroup of the affine extension of a permutation representation.
    
    \begin{defn} \label{Def: cocycle}
        A \textbf{1-cocycle} is a function $\gamma:G\to \F^X$ such that for any $g,h\in G$ we have $$\gamma(gh)=\gamma(g)+A_g(\gamma(h)).$$ The set of 1-cocycles forms an $\F$-vector space, which we denote by $Z^1(G,\F^X)$.
    \end{defn}
    
    \begin{defn} \label{Def: coboundary}
        A \textbf{1-coboundary} $\gamma:G\to \F^X$ is a 1-cocycle such that there exists some $t\in \F^X$, such that for all $g\in G$ we have $$\gamma(g)=(I-A_g)(t).$$ The set of 1-coboundaries forms an $\F$-subvector space of $Z^1(G,\F^X)$, and is denoted by $B^1(G,\F^X)$. 
    \end{defn}

    The 1-coboundaries occur as a result of global conjugation, as is shown in Section \ref{Section: conjugacy results}. In a similar fashion, we mimic `local' conjugation via the use of a local 1-coboundary. 
    
    \begin{defn} \label{Def: local coboundary}
        A \textbf{local 1-coboundary} is a 1-cocycle $\gamma:G\to \F^X$ such that for all $g\in G$ there exists some $t_g\in \F^X$ such that $$\gamma(g)=(I-A_g)(t_g).$$
    \end{defn}

    Local 1-coboundaries have several nice properties for permutation representations.

    \begin{lem} \label{Lem: properties of local coboundaries}
        Let $\gamma:G\to \F^X$ be a local 1-coboundary, and let $x\in X$. The following hold:

        \begin{itemize}
            \item[(i)] We have $\gamma(\id)=\zero\in \F^X$.
            \item[(ii)] For any $g\in G$ such that $g(x)=x$, we have $\gamma(g)(x)=0$.
            \item[(iii)] For any $g,h\in G$ such that $g(x)=h(x)$, we have $\gamma(g^{-1})(x)=\gamma(h^{-1})(x)$.
        \end{itemize}
    \end{lem}

    \begin{proof}
        For this proof, let $t_g\in \F^X$ such that for any $g\in G$ we have $\gamma(g)=(I-A_g)(t_g)$.
    
        \textbf{(i)} We write out $$\gamma(\id)=(I-A_{\one})(t_{\one})=(I-I)(t_{\one})=\zero.$$

        \textbf{(ii)} Suppose $g(x)=x$. Then $$\gamma(g)(x)=(I-A_g)(t_g)(x)=t_g(x)-A_g(t_g)(x)=t_g(x)-t_g(g^{-1}(x))=t_g(x)-t_g(x)=0.$$

        \textbf{(iii)} Suppose $g(x)=h(x)$, so that $g^{-1}h(x)=x$. We get \begin{align}
            0 &=\gamma(g^{-1}h)(x) \label{Eq: properties of local coboundary 1}\\ &=\gamma(g^{-1})(x)+A_{g^{-1}}(\gamma(h))(x)\label{Eq: properties of local coboundary 2}\\ &=\gamma(g^{-1})(x)+\gamma(h)(g(x))\nonumber\\
            &=\gamma(g^{-1})(x)+\gamma(h)(h(x))+\gamma(h^{-1})(x)-\gamma(h^{-1})(x)\nonumber\\
            &=\gamma(g^{-1})(x)+A_{h^{-1}}(\gamma(h))(x)+\gamma(h^{-1})(x)-\gamma(h^{-1})(x)\nonumber\\
            &= \gamma(g^{-1})(x)+\gamma(h^{-1}h)(x)-\gamma(h^{-1})(x)\label{Eq: properties of local coboundary 3}\\
            &=\gamma(g^{-1})(x)-\gamma(h^{-1})(x)\label{Eq: properties of local coboundary 4}
        \end{align}

        where (\ref{Eq: properties of local coboundary 1}) follows from (ii), (\ref{Eq: properties of local coboundary 2}) and (\ref{Eq: properties of local coboundary 3}) follow from Definition \ref{Def: cocycle}, and (\ref{Eq: properties of local coboundary 4}) follows from (i).
    \end{proof}

    In fact, when $G$ is a subgroup of an affine extension of a permutation representation and $\gamma$ is a local 1-coboundary of $G$, it turns out that $\gamma$ is always a (global) 1-coboundary.

    The following result can be found to be true by application of the (Eckmann-)Shapiro lemma for the cohomology group $H^1(G,\F^X):=Z^1(G,\F^X)/B^1(G,\F^X)$, see for example \cite[Chapter 6.3]{Weibel}. However, the result is not immediately clear, and one would need to read through the technical machinery of category theory to understand that the results below are in fact true. Instead, we provide a more direct proof of the fact.

    \begin{thm} \label{Thm: local coboundary is global coboundary}
        Let $G\leq \F^X\rtimes_\rho P$ be a subgroup. Any local 1-coboundary $\gamma:G\to \F^X$ is a 1-coboundary. 
    \end{thm}

    \begin{proof}
        Let $\gamma:G\to \F^X$ be a local 1-coboundary, and split $X$ into $G$-orbits $X_i$, so that $$\F^X=\bigoplus_{i=1}^r \F^{X_i}.$$ For each $x\in X$, we let $G(x)\subseteq X$ denote its $G$-orbit. For $i=1,...,r$, choose representatives $x_i\in X_i$ of these orbits, and for each $x\in X$ choose $g_x\in G$ such that $$g_x(x_i)=x,$$ where $i$ is chosen so that $G(x_i)=G(x)$. Using (\ref{Eq: A_g acting on delta_x}), we get $$A_{g_x^{-1}}(\delta_{x_i})=\delta_{x}.$$ 
        
        Define $t\in \F^X$ by $$t(x):=-\gamma(g_x^{-1})(x_i),$$ where again $i$ is chosen such that $G(x_i)=G(x)$. We note that this is uniquely defined (well-defined), since for any $g,h\in G$ with $g(x_i)=h(x_i)=x$ we have $\gamma(g^{-1})(x_i)=\gamma(h^{-1})(x_i)$ by Lemma \ref{Lem: properties of local coboundaries}(iii).
        
        Next, let $g\in G$, and let $x,x_i\in X$ so that $x_i$ is the representative of the orbit $G(x)$. If we evaluate $(I-A_g)(t)\in \F^X$ at $x$, we get: \begin{align}
            \nonumber\left((I-A_g)(t)\right)(x)&=(t-A_g(t))(x)=t(x)-t(g^{-1}(x))=-\gamma(g^{-1}_x)(x_i)+\gamma(g^{-1}_{g^{-1}(x)})(x_i)\\ &\label{Eq: local-global solution1}=-\gamma(g^{-1}_x)(x_i)+\gamma(g^{-1}_xg)(x_i)\\ &=\label{Eq: local-global solution2}-\gamma(g^{-1}_x)(x_i)+\gamma(g^{-1}_x)(x_i)+A_{g^{-1}_x}(\gamma(g))(x_i)\\ &= \nonumber\gamma(g)(g_x(x_i))=\gamma(g)(x),
        \end{align} where (\ref{Eq: local-global solution1}) follows from Lemma \ref{Lem: properties of local coboundaries}(iii) since $$g^{-1}g_x(x_i)=g^{-1}(x)=g_{g^{-1}(x)}(x_i),$$ and where (\ref{Eq: local-global solution2}) follows from Definition \ref{Def: cocycle}. It follows that we have $$\gamma(g)=(I-A_g)(t),$$ proving that $\gamma$ is a 1-coboundary.
    \end{proof}

\section{Local-global conjugacy results for affine extensions} \label{Section: conjugacy results}

    In this section, we prove Theorem \ref{Thm: IntroMain1} and similar results as consequences to the result in Section \ref{Section: group cohomology}. Just like the previous section, we note that it is not needed for the field $\F$ to be finite for the results of this section to hold, since this condition is not used throughout the proofs.

    \subsection{Local-global conjugacy for trivial kernels} \label{Subsection: Local-global conjugacy for trivial kernels}

    The following theorem relates the local-global conjugacy problem to the study of group cohomology.

    \begin{thm} \label{Thm: Goal theorem1}
        Let $H,G\leq \F^X\rtimes_\rho P$ be subgroups. Then $H$ is globally $K$-conjugate into $G$ if and only if there exist $v_h\in K$ for $h\in H$ such that the following hold:

        \begin{itemize}
            \item[(i)] For any $h\in H$ we have $v_hhv_h^{-1}\in G$;
            \item[(ii)] For any $g,h\in H$ we have $$(v_ggv_g^{-1})(v_hhv_h^{-1})=v_{gh}ghv_{gh}^{-1}.$$ 
        \end{itemize}
    \end{thm}

    \begin{proof}
        If $vHv^{-1}\subseteq G$ for some $v\in K$, then conditions (i),(ii) follow immediately by choosing $v_h=v$ for all $h\in H$. Therefore suppose that conditions (i),(ii) hold for certain elements $v_h\in K$ for $h\in H$. Write $v_h=(t_h,\one)\in \F^X\times \{\one\}$ for $t_h\in \F^X$, and note that, by (\ref{Eq: conjugating g by some t}) and (i), we have $$((I-A_h)(t_h),\one)h\in G$$ for all $h\in H$. With (\ref{Eq: conjugating g by some t}) in mind, (ii) changes to: \begin{align*}
            ((I-A_g)(t_g)+A_g(I-A_h)(t_h),\one)gh&=((I-A_g)(t_g),\one)g((I-A_h)(t_h),\one)h\\&=((I-A_{gh})(t_{gh}),\one)gh
        \end{align*} where the first equality follows from (\ref{Eq: permutation representation}). Hence, we have \begin{align}\label{Eq: cocycle condition in conjugation}
            (I-A_g)(t_g)+A_g(I-A_h)(t_h)=(I-A_{gh})(t_{gh})
        \end{align} Therefore, if we let $\gamma:H\to \F^X$ denote the function given by $h\mapsto (I-A_h)(t_h)$, (\ref{Eq: cocycle condition in conjugation}) changes to $$\gamma(gh)=\gamma(g)+A_g(\gamma(h)).$$ That is, $\gamma$ is a 1-cocycle in the sense of Definition \ref{Def: cocycle}. Clearly, it also satisfies Definition \ref{Def: local coboundary}, and thus by Theorem \ref{Thm: local coboundary is global coboundary}, there exists a single $t\in \F^X$ such that $\gamma(h)=(I-A_h)(t)$. Choose $v:=(t,\one)\in K$. Then, using (\ref{Eq: conjugating g by some t}), for any $h\in H$ we have: \begin{align*}
            vhv^{-1}&=(t,\one)h(t,\one)^{-1}=((I-A_h)(t),\one)h\\&=(\gamma(h),\one)h=((I-A_h)(t_h),\one)h=v_hhv_h^{-1}\in G.
        \end{align*} That is, $vHv^{-1}\subseteq G$. This completes the proof.
    \end{proof}

    Note that condition (i) of Theorem \ref{Thm: Goal theorem1} is given by elementwise conjugation. Using this, we can prove Theorem \ref{Thm: IntroMain1} by forcing condition (ii) of Theorem \ref{Thm: Goal theorem1}. 

    \begin{proof}[Proof of Theorem \ref{Thm: IntroMain1}]
        Suppose $H$ is elementwise $K$-conjugate into $G$. For all $h\in H$, choose $v_h\in K$ such that $v_hhv_h^{-1}\in G$. For any pair $g,h\in H$, note that $$\pi((v_ggv_g^{-1})(v_hhv_h^{-1})(v_{gh}ghv_{gh}^{-1})^{-1})=\pi(gh(gh)^{-1})=\one\in P.$$ Hence since $G\cap K=\{(\zero,\one)\}$, we get $$(v_ggv_g^{-1})(v_hhv_h^{-1})=v_{gh}ghv_{gh}^{-1}.$$ That is, the chosen elements $v_h$ satisfy conditions (i),(ii) of Theorem \ref{Thm: Goal theorem1}. Hence $H$ is globally $K$-conjugate into $G$. 
    \end{proof}

    We highlight some results that follow immediately from Theorem \ref{Thm: IntroMain1}.

    \begin{example}
        Theorem \ref{Thm: Goksel's proven properties}(b) follows from Theorem \ref{Thm: IntroMain1} by choosing $X=V_{n-1}$ and $P=\Tnn$ as in Section \ref{Section: Preliminaries: rooted binary trees explanation}.
    \end{example}    

    \begin{example}
        In addition, Theorem \ref{Thm: IntroMain1} proves the statement for any iterated wreath product $$[C_p]^n:=C_p\wr C_p\wr \cdots \wr C_p$$ of the cyclic groups $C_p\cong \Z/p\Z$ for $p$ prime, since these groups admit the recursive relation $$[C_p]^n\cong \F_p^{V^{(p)}_{n-1}}\rtimes [C_p]^{n-1},$$ where $V_{n-1}^{(p)}$ is the set of words over $\{0,...,p-1\}$ of length exactly $n-1$. 
        
        If $f$ is a unicritical degree $p$ polynomial, we know that the finite-level \emph{geometric} monodromy groups of $f$ embed into $[C_p]^n$ up to global conjugacy, see for example \cite{AdamsHyde}. Hence it makes sense to study these local-global questions for $[C_p]^n$ as well in the case of the arithmetic Galois groups. The following result summarises this.
    \end{example}

    \begin{cor} \label{Cor: Iterated C_p}
        Let $p$ be prime, and for each $n\geq 0$, let $$[C_p]^n:=C_p\wr C_p\wr \cdots \wr C_p$$ be the $n$-fold wreath product of $C_p\cong \Z/p\Z$ with itself. For each $n\geq 1$, let $$\pi_n:[C_p]^n=C_p\wr [C_p]^{n-1}\to [C_p]^{n-1}$$ denote the quotient map, and set $K_n=\ker(\pi_n)$. 
        
        If $H,G\subseteq [C_p]^n$ are subgroups such that $G\cap K_n=\{\one\}$, then $H$ is elementwise $K_n$-conjugate into $G$ if and only if $H$ is globally $K_n$-conjugate into $G$.
    \end{cor}

    The groups $C_p$ in Corollary \ref{Cor: Iterated C_p} specifically require $p$ to be prime, since this forces $[C_p]^n$ to be an affine extension. However, the proof relies only on the Abelian group structure. Thus, as noted in Remark \ref{Re: generalisation}, we can replace $C_p\cong \F_p$ by $C_d$ for any positive integer $d$. Moreover, it can readily be seen that the statement holds for iterated wreath products $[\A]^n$ for any Abelian group $\A$.

    \begin{example} \label{Example: demonstration of generalisation of Goksel}
        We demonstrate how Theorem \ref{Thm: IntroMain1} may be used to show $N$-conjugacy instead of $K$-conjugacy for some subgroups $N\subseteq K$. For the purpose of the demonstration, let $S_8$ be the group of bijections on $\{0,1,...,7\}$, and consider the affine extension $$\F^{\{0,...,7\}}\rtimes_\rho S_8,$$ where $\F$ is any field. Let $a=(0\ 5\ 3\ 6)(1\ 4\ 2\ 7)\in S_8, \  b=(0\ 3)(1\ 2)\in S_8$ and $c=(0\ 6)(1\ 7)(2\ 4)(3\ 5)\in S_8$. Let $$H:=\langle (\zero,a),(\zero,b),(\zero,c)\rangle\subseteq \F^{\{0,...,7\}}\rtimes_\rho S_8.$$ Now consider $x:=(\delta_0+\delta_5,(0\ 5)(1\ 4)(2\ 7)(3\ 6))$ and $y:= (\delta_0+\delta_3,b)$, where $\delta_i:X\to \F$ is the indicator function as in Definition \ref{Def: indicator function}, for $i=0,...,7$. Then for $G:=\langle x,y\rangle$, one can check that $H$ is elementwise $N$-conjugate into $G$, where $$N:=\langle (\delta_0,\one),(\delta_3,\one),(\delta_5,\one),(\delta_6,\one)\rangle\subseteq K.$$ Furthermore, notice that $\{0,3,5,6\}$ is invariant under the action of $G$.
        
        Next, we write $$\F^{\{0,...,7\}}\cong\F^{\{0,3,5,6\}}\rtimes\left(\F^{\{1,2,4,7\}}\rtimes \pi(G)\right).$$ Since $N\cap G=\{(\zero,\id)\}$, Theorem \ref{Thm: IntroMain1} shows that $H$ is globally $N$-conjugate into $G$. And as a matter of fact, we have $(\delta_0,\one)H(\delta_0,\one)^{-1}=G$.
    \end{example}

    \subsection{Local-global conjugacy for quotients of permutation representations} Suppose we have a group $G\leq \F^X\rtimes_\rho P$, and let $N:=G\cap K$. If $N$ is trivial, then by Theorem \ref{Thm: IntroMain1}, the properties $\p(H,G)$ hold for all $H\leq \F^X\rtimes_\rho P$. In this section, we assume that $N$ is non-trivial, and hope to apply Theorem \ref{Thm: IntroMain1} to a quotient of the representation of $G$, with the idea in mind that $G/N$ intersects $(\F^X\times \{\one\})/N$ trivially.

    Let $F\subseteq \F^X$ be the subspace defined by $N=F\times \{\one\}$, and note that we have an isomorphism \begin{align} \label{Isom: Phi}
        \Phi:(\F^X\rtimes_\rho\pi(G))/N\isomto \F^X/F\rtimes_{\rho_{G,N}} G/N, \phantom{;;;} (t,\pi(g))N\mapsto (t+F,gN).
    \end{align}

    where $\rho_{G,N}:G/N\to \text{Aut}(\F^X/F)$ is the representation \begin{align} \label{Rep: rho G N}
        \rho_{G,N}(gN):=[t+F\mapsto \rho(g)(t)+F].
    \end{align} Here we make use of the fact that any linear map $A:\F^X\to \F^X$ such that $A(F)=F$, there is an induced linear map $A_*:\F^X/F\to \F^X/F$ given by $A_*(t+F)=A(t)+F$. Let $\pi_{G,N}$ denote the restriction map $$\pi_{G,N}:(\F^X\rtimes_\rho \pi(G))/N\to G/N,\phantom{;;;} (t,\pi(g))N\mapsto gN,$$ and let $K_{G,N}:=\ker(\pi_{G,N})=\left(\F^X\times \{\one\}\right)N$. We define properties $\q(H,H')$ similar to that of Definition \ref{Def: K-conjugacy}, but we use a $\q$ instead of $\p$ to highlight that we are not (necessarily) working with a permutation representation.
    
    \begin{defn}
        Let $H,H'\leq (\F^X\rtimes_{\rho} \pi(G))/N$. We say that $H$ is \textbf{elementwise $K_{G,N}$-conjugate into $H'$} if for every $h\in H$ there exists some $f\in K_{G,N}$ such that $fhf^{-1}\in H'$. We say that $H$ is \textbf{globally $K_{G,N}$-conjugate into $H'$} if there exists some $f\in K_{G,N}$ such that $fHf^{-1}\subseteq H'$. Finally, we say that $\q(H,H')$ \textbf{holds} if:
    
        \begin{center}$H$ is elementwise $K_{G,N}$-conjugate into $H'\iff H$ is globally $K_{G,N}$-conjugate into $H'$.\end{center}
    \end{defn}

    Suppose that $H,G\leq \F^X\rtimes_\rho P$ are subgroups. If $\pi(H)\not\subseteq \pi(G)$, choose $h\in H$ such that $\pi(h)\notin \pi(G)$. Clearly, $\pi(vhv^{-1})=\pi(h)$ for all $v\in K$, so $h$ is not $K$-conjugate into $G$. Hence $\p(H,G)$ always holds; $H$ is neither elementwise nor globally $K$-conjugate into $G$.

    As a result, we only need to consider the case when $\pi(H)\subseteq \pi(G)$. In this case, we can identify $HN/N$ with a subgroup of $(\F^X\rtimes_\rho \pi(G))/N$, giving us the following result.
    
    \begin{prop} \label{Prop: Old main2}
        Let $G\leq \F^X\rtimes_\rho P$, and define $N:=G\cap (\F^X\times\{\one\})$. Let $H\leq \F^X\rtimes_\rho P$ be any subgroup such that $\pi(H)\subseteq \pi(G)$.
        
        Then, $\p(H,G)$ holds if and only if $\q(HN/N,G/N)$ holds.
    \end{prop}

    \begin{proof}
        We show that for any $h\in H$ and $v\in K$, we have $vhv^{-1}\in G$ if and only if $(vN)(hN)(vN)^{-1}\in G/N$, where $hN\in HN/N$ is the corresponding coset. 
        
        Indeed, if $vhv^{-1}\in G$, then for any $n\in N$ we have $$(vn)(hN)(vn)^{-1}=v(hN)v^{-1}=vhv^{-1}N\subseteq G,$$ where we use that $N$ is a normal subgroup of $\F^X\rtimes \pi(G)$. Hence $(vN)(hN)(vN)^{-1}\in G/N$. Similarly if $(vN)(hN)(vN)^{-1}\in G/N$, then in particular $vhv^{-1}\in \langle G,N\rangle=G$. Since these claims holds for any $h\in H$, we immediately conclude that $\p(H,G)\iff \q(HN/N,G/N)$, finishing the proof.
    \end{proof}

    Following Proposition \ref{Prop: Old main2}, we can apply Theorem \ref{Thm: IntroMain1}, as long as $\rho_{G,N}:G/N\to \aut(\F^X/F)$ becomes (isomorphic to) a permutation representation.

    \begin{proof}[Proof of Theorem \ref{Thm: IntroMain2}]
        Recall that we have an isomorphism $$\Phi:(\F^X\rtimes_\rho\pi(G))/N\isomto \F^X/F\rtimes_{\rho_{G,N}} G/N$$ given by $\Phi((t,\pi(g))N)=(t+F,gN)$. Under $\Phi$, we have $\Phi(K_{G,N})=\F^X/F\times \{\one\}$. Then, since $\phi$ is an isomorphism, we have an isomorphism $$\Psi:\F^X/F\rtimes_{\rho_{G,N}}G/N\isomto \F^Y\rtimes_\tau G/N$$ given by $(t+F,gN)\mapsto (\phi(t+F),gN)$. Indeed, by (\ref{Eq: isomorphism diagram}), for any $g\in G$ and $s\in \F^X$ we have $$\tau(gN)\circ \phi=\phi_*(\rho_{G,N}(gN))\circ \phi=\left(\phi\circ \rho_{G,N}(gN)\circ \phi^{-1}\right)\circ \phi=\phi\circ \rho_{G,N}(gN),$$ so that \begin{align*}
            \Psi((t+F,gN)(s+F,hN))&=\Psi(t+F+\rho_{G,N}(gN)(s+F),ghN)\\&=(\phi(t+F)+\phi(\rho_{G,N}(gN)(s+F)),ghN) \\
            &=(\phi(t+F)+\tau(gN)(\phi(s+F))),ghN)\\&=(\phi(t+F),gN)(\phi(s+F),hN)\\
            &=\Psi(t+F,gN)\Psi(s+F,hN).
        \end{align*}

        Under the isomorphism, we have $\Psi(\Phi(K_{G,N}))=\F^Y\times\{\one\}$, and thus $\q(HN/N,G/N)$ holds if and only if $\p_\tau(\Psi(\Phi(HN/N)),\Psi(\Phi(G/N)))$ holds, where the property $\p_\tau$ is given by Definition \ref{Def: K-conjugacy} for the permutation representation $\tau:G/N\to \aut(\F^Y)$. Since $$\Psi(\Phi(G/N))\cap \left(\F^Y\times \{\id\}\right)=\{(\zero,\id)\},$$ we know that $\p_\tau(\Psi(\Phi(HN/N)),\Psi(\Phi(G/N)))$ holds by Theorem \ref{Thm: IntroMain1}. Hence, $\p(H,G)$ holds by Proposition \ref{Prop: Old main2}.
    \end{proof}

    \begin{example} \label{Ex: stabiliser}
        As an example, let $x\in X$, and choose some subgroup $G_0\subseteq P$ such that $g(x)=x$ for all $g\in G_0$. Define an extension of $G_0$ by $$G:=\left\langle \left\{(\zero,g)\in \F^X\rtimes_\rho P\mid g\in G_0\right\}\cup\{(\delta_x,\one)\}\right\rangle\subseteq \F^X\rtimes_\rho P.$$ Since for each $g\in G$ we have $A_g(\delta_x)=\delta_x$, we can choose $Y:=X\setminus \{x\}$. For $N=G\cap K=\langle (\delta_x,\one)\rangle$, we have an induced permutation representation $$\rho_{G,N}:G/N\to \aut(\F^Y),$$ given by the action of $G$ on $X\setminus \{x\}$. Hence, for any subgroup $H\leq \F^X\rtimes_\rho P$, the property $\p(H,G)$ holds by Theorem \ref{Thm: IntroMain2}. 
    \end{example}
        
\section{The odometer flag and the counterexample} \label{Section: counterexample}

    In \cite{Goksel}, the author states Conjecture \ref{Conj: Goksel}, which, if it were true, would prove Conjecture \ref{Conj: GokselMarkov}. In particular, it would justify the use of the even Markov model for predicting the factorisations of polynomials. 
    
    Fix some positive integer $n$, and consider the first $n$ levels of the rooted binary tree $T_n$. For this section, recall the notation and definitions from Section \ref{Section: Preliminaries: rooted binary trees explanation}. In particular, we may write $$\Tn=\F_2^{V_{n-1}}\rtimes \Tnn.$$ The following notion is used throughout this section.
    
    \begin{defn} \label{Defn: n-odometer}
        An \textbf{$n$-odometer} is an element $g\in \Tn$ that acts transitively on $V_n$.
    \end{defn}

    Write $\pi_n:\Tn\to \Tnn$ for the restriction homomorphism, and let $$K_n:=\ker(\pi_n)=\F_2^{V_{n-1}}\times \{\id\}\subseteq \F_2^{V_{n-1}}\rtimes \Tnn.$$ Let $\p(H,G)$ denote property $\p_{K_n}(H,G)$ as in Definition \ref{Def: K-conjugacy}, using $P=\Tnn$, $\F=\F_2$, $X=V_{n-1}$ and $K=K_n$.

    This section is used to provide a counterexample to Conjecture \ref{Conj: Goksel}, which is stated in Section \ref{Section: Revision: Counterexample}. We hope that the preparatory results on local-global conjugacy of subgroups in $\Tn$ that contain an $n$-odometer in Sections \ref{Section: Revision: Flag} and \ref{Section: Revision: Application} are useful on their own.   

    In Section \ref{Section: Revision: Flag}, we formulate a \textit{flag} on the vector space $\F_2^{V_{n-1}}$ in the form of subspaces $$\{\zero\}=E_0\subset E_1\subset \cdots \subset E_{2^{n-1}}=\F_2^{V_{n-1}}.$$ After that, we proceed by studying the linear-algebraic properties of this flag. In particular, in Section \ref{Section: Revision: Application} we show that for any subgroup $H\subseteq \Tn$ that contains an $n$-odometer, the subgroup $K_n\cap H$ must be exactly equal to one of the terms $E_i\times \{\one\}\subseteq \F_2^{V_{n-1}}\times \{\one\}$. Using this, we find the exact conditions needed for global $K_n$-conjugacy of such group $H$ into another group $G\leq \Tn$. The counterexample is stated in Section \ref{Section: Revision: Counterexample}, justified by the results found by use of our flag, and is verified using a SageMath program.

\subsection{The odometer flag} \label{Section: Revision: Flag}

    Let $c\in \T$ denote the standard odometer, given by the recursive formula $$c=(c,\one)(0\ 1)\in \T^2\rtimes S_2.$$ Then $c$ acts on each set $V_n$ as a $2^n$-cycle. In particular, the restriction $$c|_{T_n}=(c|_{T_{n-1}},\one)(0\ 1)\in \Tnn^2\rtimes S_2=\Tn$$ is an $n$-odometer in the sense of Definition \ref{Defn: n-odometer}. For each $0\leq m\leq 2^{n-1}$, define 
    
    \begin{center}\fbox{$E_m:=\ker((I+A_{c|_{T_n}})^m)\subseteq \F_2^{V_{n-1}}$}\end{center}
    
    where $I:=A_{\one}\in R(\Tn)$ is the linear map corresponding to the identity element $\one\in \Tn$, and where by convention $(I+A_{c|_{T_n}})^0=I$. Note that addition and subtraction of these linear maps are the same, as we are working over $\F_2$. Clearly, we have inclusions $$\{\zero\}= E_0\subseteq E_1\subseteq \cdots \subseteq E_{2^{n-1}}\subseteq \F_2^{V_{n-1}}$$ of $\F_2$-vector spaces.

    \begin{lem} \label{Lem: Computing E_1,E_{2^{n-1}}}
        The following statements hold.

        \begin{itemize}
            \item $E_{2^{n-1}}=\F_2^{V_{n-1}}$.
            \item Let $\textbf{1}:V_{n-1}\to \F_2$ denote the constant function $v\mapsto 1$. Then $E_1=\langle \textbf{1}\rangle$.
        \end{itemize}
    \end{lem}

    \begin{proof}
        Since $R(\Tn)$ is an $\F_2$-algebra, we have the natural trick $$(I+A_{c|_{T_n}})^{2^{n-1}}=I^{2^{n-1}}+A_{c|_{T_n}}^{2^{n-1}}=I+A_{c|_{T_n}^{2^{n-1}}},$$ where the last equality follows from property (\ref{Eq: linear maps commute with group}). Then note that ${c|_{T_n}}$ must act as a $2^{n-1}$-cycle on $V_{n-1}$, so that the induced map $$c|_{T_n}^{2^{n-1}}:V_{n-1}\to V_{n-1}$$ is trivial. Hence, if $t:V_{n-1}\to \F_2$ is a function, then $$A_{c|_{T_n}^{2^{n-1}}}(t)=t\circ (c|_{T_n}^{2^{n-1}})^{-1}=t,$$ so that $A_{c|_{T_n}^{2^{n-1}}}=I$. It follows that $(I+A_{c|_{T_n}})^{2^{n-1}}=0$, so that $E_{2^{n-1}}=\F_2^{V_{n-1}}$.

        Now let $t\in \ker(I+A_{c|_{T_n}})$. By definition this means that $$\zero=(I+A_{c|_{T_n}})(t)=t+t\circ c|_{T_n}^{-1},$$ and hence $t(v)=t(c|_{T_n}^{-1}(v))$ must hold for any word $v\in V_{k-1}$. As a result, we find $$t(v)=t(c|_{T_n}^{-1}(v))=t(c|_{T_n}^{-2}(v))=\cdots =t(c|_{T_n}^{-2^n+1}(v)).$$ Since ${c|_{T_n}}$ acts transitively on $V_{n-1}$, this implies that $t$ must be constant.
    \end{proof}

    Next we show that the inclusions $E_m\subseteq E_{m+1}$ have codimension 1.

    \begin{cor} \label{Cor: flag properties}
        For each $0< m \leq 2^{n-1}$, the following are true.

        \begin{itemize}
            \item[(i)] We have $\dim E_m=m$.
            \item[(ii)] We have $E_{m-1}=(I+A_{c|_{T_n}})E_m$.
            \item[(iii)] For any $t\in E_{2^{n-1}}\setminus E_{2^{n-1}-1}$ and any $0\leq m\leq 2^{n-1}$ we have $$E_m=\left\langle \left(I+A_{c|_{T_n}}\right)^{2^{n-1}-1}(t),\left(I+A_{c|_{T_n}}\right)^{2^{n-1}-2}(t),...,\left(I+A_{c|_{T_n}}\right)^{2^{n-1}-m}(t)\right\rangle.$$
        \end{itemize}
    \end{cor}

    \begin{proof}
        \textbf{(i).} Note that, for each $k$, $I+A_{c|_{T_n}}$ induces a linear map $$I+A_{c|_{T_n}}:E_k=\ker((I+A_{c|_{T_n}})^k)\to \ker((I+A_{c|_{T_n}})^{k-1})=E_{k-1}.$$ By the rank-nullity theorem, we get \begin{align}\label{Eq:rank-nullity}\dim E_k\leq  \dim E_{k-1}+\dim(\ker(I+A_{c|_{T_n}})\cap E_k)\leq \dim E_{k-1}+\dim (\ker(I+A_{c|_{T_n}})).\end{align}
    
        By Lemma \ref{Lem: Computing E_1,E_{2^{n-1}}}, we know that $\dim E_1=1$, so that by (\ref{Eq:rank-nullity}) we get \begin{align} \label{Eq: dimension of flag increases by 1}
            \dim E_k\leq 1+\dim E_{k-1}
        \end{align}
        for all $0<k\leq 2^{n-1}$. Since by Lemma \ref{Lem: Computing E_1,E_{2^{n-1}}} we have $\dim E_{2^{n-1}}=2^{n-1}$, the formula of (\ref{Eq: dimension of flag increases by 1}) necessarily implies that $\dim E_m=m$ for each $0\leq m\leq 2^{n-1}$.
        
        \textbf{(ii).} Since $$\ker (I+A_{c|_{T_n}})=E_1\subseteq E_m$$ has dimension 1, the maps $I+A_{c|_{T_n}}:E_m\to E_{m-1}$ must be surjective, and thus $$E_{m-1}=(I+A_{c|_{T_n}})E_m.$$

        \textbf{(iii).} Next, let $t\in E_{2^{n-1}}\setminus E_{2^{n-1}-1}$, and define $t_m$ for $0\leq m\leq 2^{n-1}$ by $t_{2^{n-1}}:=t$, and iteratively $$t_{m-1}:=(I+A_{c|_{T_n}})(t_m).$$ If we assume that $t_m\in E_m\setminus E_{m-1}$, which is true by assumption for $m=2^{n-1}$, then $$(I+A_{c|_{T_n}})^{m-2}t_{m-1}=(I+A_{c|_{T_n}})^{m-2}(I+A_{c|_{T_n}})t_m=(I+A_{c|_{T_n}})^{m-1}t_m\neq 0,$$ shows that $t_{m-1}\in E_{m-1}\setminus E_{m-2}$. Therefore by (i) and (ii) we get $E_m=\langle E_{m-1},t_m\rangle$, proving that $$E_m=\langle t_1,...,t_m\rangle.$$
    \end{proof}

    \begin{defn} \label{Def: Flag}
        We call the sequence $$\{\zero\}=E_0\subset E_1\subset \cdots \subset E_{2^{n-1}}=\F_2^{V_{n-1}}$$ the \textbf{(odometer) flag of $\F_2^{V_{n-1}}$}.
    \end{defn}

    \begin{defn} \label{Def: the functions e_m}
        Let $\delta_{0^{n-1}}:V_{n-1}\to \F_2$ denote the indicator function of the word $0^{n-1}$ as in Definition \ref{Def: indicator function}. For each $0\leq k\leq 2^{n-1}$, we define 

        \begin{center}
            \fbox{$e_{2^{n-1}-k}:=(I+A_{c|_{T_n}})^k(\delta_{0^{n-1}}).$}
        \end{center}
    \end{defn}

    \begin{lem} \label{Lem: e_2 in E_2}
        We have $e_{2^{n-1}}\in E_{2^{n-1}}\setminus E_{2^{n-1}-1}$. As a consequence, for every $0< m\leq 2^{n-1}$, we have $$E_m=\langle e_1,...,e_m\rangle=\langle E_{m-1},e_m\rangle.$$
    \end{lem}

    \begin{proof}
        We use the following claim.

        \begin{claim} \label{Claim: E_m is invariant under A_g}
            For any $m$, we have $A_{c|_{T_n}}(E_m)\subseteq E_m$.
        \end{claim}

        \begin{proof}
            Let $t\in E_m=\ker((I+A_{c|_{T_n}})^m)$. Then $$(I+A_{c|_{T_n}})^m(A_{c|_{T_n}}(t))=A_{c|_{T_n}}(I+A_{c|_{T_n}})^m(t)=A_{c|_{T_n}}(\zero)=\zero,$$ where we use that $A_{c|_{T_n}}$ commutes with $I+A_{c|_{T_n}}$. Hence $A_{c|_{T_n}}(t)\in E_m$.
        \end{proof}
    
        Now suppose that $e_{2^{n-1}}\in E_m$ for some integer $m$. For any integer $k$, we have $$A_{c|_{T_n}}^k(e_{2^{n-1}})(c|_{T_n}^k(0^{n-1}))=e_{2^{n-1}}(c|_{T_n}^{-k}c|_{T_n}^k(0^{n-1}))=1,$$ and for $v\in V_{n-1}\setminus \{0^{n-1}\}$ we have $$A_{c|_{T_n}}^k(e_{2^{n-1}})(c|_{T_n}^{k}(v))=e_{2^{n-1}}(v)=0,$$ so we know that $$A_{c|_{T_n}}^k(e_{2^{n-1}})=\delta_{c|_{T_n}^k(0^{n-1})}.$$ Since ${c|_{T_n}}$ acts transitively on $V_{n-1}$, this means that we obtain all $\delta_u$ for $u\in V_{n-1}$ by varying $k$, so we obtain the standard basis elements of $\F_2^{V_{n-1}}$. By Claim \ref{Claim: E_m is invariant under A_g}, we know that $A_{c|_{T_n}}^k(e_{2^{n-1}})\in E_m$ for each $k$, and therefore $E_m=\F_2^{V_{n-1}}=E_{2^{n-1}}$ must hold.

        That $E_m=\langle e_1,...,e_m\rangle =\langle E_{m-1},e_m\rangle$ holds follows immediately from Corollary \ref{Cor: flag properties}(iii).
    \end{proof}

    Recall the following notion from linear algebra.

    \begin{defn}
        Let $F$ be a vector space, and $g\in \text{End}(F)$. A subspace $W\subseteq F$ is called \textbf{$g$-invariant}, if for any $w\in W$ we have $g(w)\in W$.
    \end{defn}

    \begin{prop} \label{Prop: smallest A_g-invariant space is E_m}
        Let $1\leq m\leq 2^{n-1}$. For any $e\in E_m\setminus E_{m-1}$, the smallest $A_{c|_{T_n}}$-invariant subspace of $E_m$ that contains $e$ equals $E_m$.
    \end{prop}

    \begin{proof}
        Using Lemma \ref{Lem: e_2 in E_2}, write $e=e_m+a_{m-1}e_{m-1}+\cdots +a_1e_1$ for $a_i\in \F_2$, and let $E\subseteq E_m$ be the smallest $A_{c|_{T_n}}$-invariant subspace that contains $e$. In particular, the map $I+A_{c|_{T_n}}$ must restrict to a linear map on $E$ as well. For each $0\leq k<m$, let \begin{align} \label{Eq: e_2 in E_2 equation1}
            t_k:=(I+A_{c|_{T_n}})^{m-k}(e)=e_{k}+a_{m-1}e_{k-1}+\cdots + a_{m-k+1}e_1\in E.
        \end{align} It follows that $e_1=t_1\in E$. If we assume that $\langle e_1,...,e_{k-1}\rangle\subseteq E$ as an induction hypothesis, then by (\ref{Eq: e_2 in E_2 equation1}) we also have $$t_k-(a_{m-1}e_{k-1}+\cdots +a_{m-k+1}e_1)=e_k\in E,$$ proving the induction. Hence $E=\langle e_1,...,e_m\rangle$, finishing the proof. 
    \end{proof}

    \begin{prop} \label{Prop: for all m,h, A_h(E_m) = E_m}
        Let $f\in \Tn$. Then, for any $0\leq m\leq 2^{n-1}$, the linear map $A_f:\F_2^{V_{n-1}}\to \F_2^{V_{n-1}}$ restricts to an isomorphism $$A_f:E_m\isomto E_m.$$
    \end{prop}

    \begin{proof}
        We show this by induction. To separate the levels $n$ and $n-1$, we let $$\{\zero\}= D_0\subset D_1\subset \cdots \subset D_{2^{n-2}}=\F_2^{V_{n-2}}$$ denote the odometer flag as in Definition \ref{Def: Flag} for $\F_2^{V_{n-2}}$. If $n=1$, then $\F^{V_{n-1}}\cong \F$, so that $E_m=\{\zero\}$ or $E_m=\F$, and in either case the linear maps $A_f:\F\to \F$ must restrict to an isomorphism on $E_m$. Therefore, assume by induction that each vector space $D_k$ is $A_h$-invariant for all $h\in \Tnn$.
        
        Define $d_{2^{n-2}}:V_{n-2}\to \F_2$ by $$\Phi(e_{2^{n-1}})=(d_{2^{n-2}},\zero),$$ where $\Phi:\F_2^{V_{n-1}}\to \F_2^{V_{n-2}}\oplus \F_2^{V_{n-2}}$ is the isomorphism from Proposition \ref{Prop: blueprint for action}, and where $e_{2^{n-1}}$ is the function from Definition \ref{Def: the functions e_m}. Note that $d_{2^{n-2}}=\delta_{0^{n-2}}$ as in Definition \ref{Def: indicator function}. Define $d_m:=(I+A_{c|_{T_n}})(d_{m+1})$ likewise, so that by Lemma \ref{Lem: e_2 in E_2} we have $D_k=\langle d_1,...,d_k\rangle$ for any $0\leq k\leq 2^{n-2}$. Finally, recall that for $g\in \Tn$ and $t,s\in \F^{V_{n-2}}$, we defined $$A_g(t,s):=\Phi(A_g(\Phi^{-1}(t,s))).$$
        
        Writing ${c|_{T_n}}=({c|_{T_{n-1}}},\one)(0\ 1)$, Proposition \ref{Prop: blueprint for action} shows that for any $t:V_{n-2}\to \F_2$ we have \begin{align}\label{Eq: Prop: E_m invariant 1}
            A_{c|_{T_n}}(t,\zero)=A_{({c|_{T_{n-1}}},\one)}A_{(0\ 1)}(t,\zero)=A_{({c|_{T_{n-1}}},\one)}(\zero,t)=(A_{c|_{T_{n-1}}}(\zero),t)=(\zero,t),
        \end{align} and \begin{align} \label{Eq: Prop: E_m invariant 2}
            A_{c|_{T_n}}(t,t)=A_{({c|_{T_{n-1}}},\one)}A_{(0\ 1)}(t,t)=A_{({c|_{T_{n-1}}},\one)}(t,t)=(A_{c|_{T_{n-1}}}(t),t).
        \end{align} Combining (\ref{Eq: Prop: E_m invariant 1}) and (\ref{Eq: Prop: E_m invariant 2}) shows $$(I+A_{c|_{T_n}})(t,\zero)=(t,\zero)+(\zero,t)=(t,t),$$ and $$(I+A_{c|_{T_n}})(t,t)=(t,t)+(A_{c|_{T_{n-1}}}(t),t)=((I+A_{c|_{T_{n-1}}})(t),\zero).$$ Therefore, if we assume $\Phi(e_{2m})=(d_m,0)$ for some $m$, which we know is true for $m=2^{n-2}$, then $$\Phi(e_{2m-1})=\Phi((I+A_{c|_{T_n}})e_{2m})=(I+A_{c|_{T_n}})(d_m,\zero)=(d_m,d_m),$$ and $$\Phi(e_{2m-2})=(I+A_{c|_{T_n}})(d_m,d_m)=((I+A_{c|_{T_{n-1}}})(d_m),\zero)=(d_{m-1},\zero).$$ We conclude that $\Phi(e_{2m})=(d_m,\zero)$ and $\Phi(e_{2m-1})=(d_m,d_m)$ for any $0<m\leq 2^{n-2}$.

        Next, we prove the induction hypothesis for the generators of $\Tn$. First consider $f=(0\ 1)=(\one,\one)\cdot (0\ 1)\in \Tnn^2\rtimes S_2$. Then from Proposition \ref{Prop: blueprint for action}, we get: $$A_f(e_{2m})=\Phi^{-1}(A_f(d_m,0))=\Phi^{-1}(0,d_m)=\Phi^{-1}((d_m,0)+(d_m,d_m))=e_{2m}+e_{2m-1}.$$ Similarly, $$A_f(e_{2m-1})=\Phi^{-1}A_f(d_m,d_m)=\Phi^{-1}(d_m,d_m)=e_{2m-1}.$$ Hence, using $E_m=\langle e_1,...,e_m\rangle$ from Lemma \ref{Lem: e_2 in E_2}, we get $$A_f(E_m)=A_f(\langle e_1,...,e_m\rangle)\subseteq \langle e_1,...,e_m\rangle =E_m.$$

        Now consider any $$f=(\one,h):=(\one,h)\cdot \one\in\Tnn^2\rtimes S_2\cong \Tn$$ for some $h\in \Tnn$. Note that by induction $A_h:D_m\to D_m$ is an isomorphism for each $m$, and therefore since $D_m=\langle d_1,...,d_m\rangle=\langle D_{m-1},d_m\rangle$, we must have $$A_h(d_m)=d_m+a$$ for some $a\in D_{m-1}$. Using Proposition \ref{Prop: blueprint for action} we compute: $$A_f(e_{2m})=\Phi^{-1}(A_f(d_m,\zero))=\Phi^{-1}(d_m,A_h(\zero))=\Phi^{-1}(d_m,\zero)=e_{2m},$$ and \begin{align*}
            A_f(e_{2m-1})&=\Phi^{-1}(A_f(d_m,d_m))=\Phi^{-1}(d_m,A_h(d_m))=\Phi^{-1}(d_m,d_m+a)\\
            &=e_{2m-1}+\Phi^{-1}(\zero,a)
        \end{align*}

        It follows that $A_f(e_{2m})\in E_{2m}$, and to show $A_f(e_{2m-1})\in E_{2m-1}$ it suffices to show that $\Phi^{-1}(\zero,a)\in E_{2m-2}$. But $a\in D_{m-1}=\langle d_1,...,d_{m-1}\rangle$, and $$e_{2m-2k}+e_{2m-2k-1}=\Phi^{-1}(\zero,d_{m-k}),$$ so $$\Phi^{-1}(\zero,a)\in \langle e_{2m-2k}+e_{2m-2k-1}: k\geq 1\rangle\subseteq E_{2m-2}.$$ It follows that $A_{(\one,h)}(E_m)=E_m$ for all $m$. Finally, note that $\Tn=\Tnn^2\rtimes S_2$ is generated by $\{(0\ 1)\}\cup \{(\one,h)\mid h\in \Tnn\}$, and hence by the rule $A_fA_h=A_{fh}$, we know that each space $E_m$ is invariant under the action of $A_h$ for all elements $f\in \Tn$, finishing the induction.
                
    \end{proof}

    \begin{lem} \label{Lem: A_h always reduces a level}
        For every $h\in \Tn$ and any $t\in E_m$, we have $(I+A_h)(t)\in E_{m-1}$. 
    \end{lem}

    \begin{proof}
        By Proposition \ref{Prop: for all m,h, A_h(E_m) = E_m}, we know that $A_h:E_k\isomto E_k$ is an isomorphism for each $k$. In particular, if $t\in E_{m-1}$, then $$(I+A_h)(t)=t+A_h(t)\in E_{m-1}.$$ Therefore, assume that $t\in E_m\setminus E_{m-1}$, so that $t= e_m+a$ for some $a\in E_{m-1}$. We must have $A_h(e_m)=e_m+b$ for some $b\in E_{m-1}$, since else the map $A_h:E_m\to E_m$ is not surjective. Therefore: $$(I+A_h)(t)=e_m+a+e_m+b+A_h(a)=a+b+A_h(a)\in E_{m-1},$$ finishing the proof.
    \end{proof}

\subsection{Application to local-global conjugacies} \label{Section: Revision: Application}

In this section, we return to the notation $$\Tn=\F_2^{V_{n-1}}\rtimes \Tnn$$ to phrase the $K_n$-conjugacy questions, and are interested in groups that contain an $n$-odometer. We use the following result.

\begin{prop}[\cite{Pink2013}, Proposition 1.6.2] \label{Prop: any pair of odometers are conjugate}
    Any two $n$-odometers are $\Tn$-conjugate.
\end{prop}

Using the theory of Section \ref{Section: Revision: Flag}, we are able to prove Theorem \ref{IntroThm: counter}.

\begin{proof}[Proof of Theorem \ref{IntroThm: counter}]
    If $H\subseteq G$, then $H$ is globally $K_n$-conjugate into $G$, so conversely suppose that $H$ is globally $K_n$-conjugate into $G$. By Proposition \ref{Prop: any pair of odometers are conjugate}, we can choose $w\in \Tn$ such that $c|_{T_n}\in wHw^{-1}\cap wGw^{-1}$, where $c$ denotes the standard odometer from Section \ref{Section: Revision: Flag}. Note that, if $$(t,\one)\in \F_2^{V_{n-1}}\rtimes \Tnn=\Tn,$$ then by (\ref{Eq: conjugating t by some g}) we have $$w(t,\one)H(t,\one)^{-1}w^{-1}=w(t,\one)w^{-1}wHw^{-1}w(t,\one)^{-1}w^{-1}=(A_w(t),\one)wHw^{-1}(A_w(t),\one)^{-1},$$ so that $H$ is globally $K_n$-conjugate into $G$ if and only if $wHw^{-1}$ is globally $K_n$-conjugate into $wGw^{-1}$. Therefore, we may assume that $w=\id$, and thus $c|_{T_n}\in H\cap G$.

    Choose $v\in K_n$ such that $vHv^{-1}\subseteq G$, and write $v=(t,\id)\in \F_2^{V_{n-1}}\rtimes \Tnn$ for some $t:V_{n-1}\to \F_2$. Let $m$ denote the integer such that $t\in E_{m+1}\setminus E_m$, so that $$(I+A_{c|_{T_n}})(t)\in E_m\setminus E_{m-1}.$$ Since by (\ref{Eq: conjugating t by some g}) we have $$((I+A_{c|_{T_n}})(t),\id)=(t,\id)c|_{T_n}(t,\id)^{-1}c|_{T_n}^{-1}\in G,$$ we know by Proposition \ref{Prop: smallest A_g-invariant space is E_m} that $E_m\times \{\one\}\subseteq G$. Let $h\in H$. Then by (\ref{Eq: conjugating t by some g}) we get \begin{align} \label{Eq: main2, number 1}
        vhv^{-1}=(t,\id)h(t,\id)^{-1}=(t+A_h(t),\id)\cdot h=((I+A_h)(t),\id)h\in G.
    \end{align} Since $(I+A_h)(t)\in E_m$ by Lemma \ref{Lem: A_h always reduces a level}, we get $((I+A_h)(t),\id)\in G$, so that we must have $h\in G$ as well by view of (\ref{Eq: main2, number 1}).
\end{proof}

On top of Theorem \ref{IntroThm: counter}, the following result describes that the structure of $H\cap K_n$ whenever $H$ contains an $n$-odometer.

\begin{prop} \label{prop: final main0}
    Let $H\leq \Tn$ be a subgroup that contains an $n$-odometer $g$. Then $$H\cap K_n=E_m\times \{\one\},$$ where $m$ is the largest integer such that $(H\cap K_n)\cap (E_m\setminus E_{m-1}\times \{\one\})\neq\emptyset$.
\end{prop}

\begin{proof}
    By Proposition \ref{Prop: any pair of odometers are conjugate}, we can choose some $w\in \Tn$ such that $wHw^{-1}$ contains $c|_{T_n}$, where $c|_{T_n}$ is the restriction of the standard odometer $c$ as in Section \ref{Section: Revision: Flag}.
    
    Let $E\subseteq \F_2^{V_{n-1}}$ denote the subspace such that $$wHw^{-1}\cap K_n=E\times \{\one\}\subseteq \F_2^{V_{n-1}}\rtimes \Tnn,$$ and let $m$ denote the largest integer such that $E\cap (E_m\setminus E_{m-1})\neq \emptyset$. Choose $t\in E$. In particular, by (\ref{Eq: conjugating t by some g}) we have $$c|_{T_n}(t,\one)c|_{T_n}^{-1}=(A_{c|_{T_n}}(t),\one)\in wHw^{-1},$$ so that $A_{c|_{T_n}}(t)\in E$. In particular, $E$ must be closed under the action of $A_{c|_{T_n}}$. By Proposition \ref{Prop: smallest A_g-invariant space is E_m}, we get $E=E_m$. Finally, since by (\ref{Eq: conjugating t by some g}) we also get $wK_nw^{-1}=K_n$, we see that $$H\cap K_n=w^{-1}(wHw^{-1}\cap K_n)w=w^{-1}(E_m\times \{\one\})w=A_{w^{-1}}(E_m)\times \{\one\}=E_m\times \{\one\}$$ by Proposition \ref{Prop: for all m,h, A_h(E_m) = E_m}.
\end{proof}

If $H,G\leq \Tn$ are subgroups, and $h\in H$ is elementwise $K_n$-conjugate into $G$, then by (\ref{Eq: conjugating t by some g}) we know that there exists some $t:V_{n-1}\to \F_2$ such that $$(t,\one)\cdot h\in G\subseteq \F_2^{V_{n-1}}\rtimes \Tnn.$$ The following theorem highlights what these elements would look like in the case that $H$ contains an $n$-odometer.

\begin{thm} \label{Thm: final main2}
    Let $H,G\subseteq \Tn=\F_2^{V_{n-1}}\rtimes \Tnn$ be subgroups such that $H$ is elementwise $K_n$-conjugate into $G$, and suppose that there exists an $n$-odometer $g\in \Tn$ such that $g\in H\cap G$. For each $h\in H$, define $$C_h:=\{t:V_{n-1}\to \F_2\mid (t,\one)h\in G\}.$$
    
    Then, for every $h\in H$, either $h\in G$, or there exists an integer $m$ such that $C_h\subseteq E_m\setminus E_{m-1}$.
\end{thm}

\begin{proof}
    Let $h\in H$. We prove that, if $C_h\not\subseteq E_m\setminus E_{m-1}$, then $h\in G$, which finishes the proof of the theorem. Hence suppose we have $t,s\in C_h$ such that $t\in E_m\setminus E_{m-1}$ and $s\in E_k\setminus E_{k-1}$ for $m> k$ without loss of generality. That is, we have $(t,\one)h,(s,\one)h\in G$. Therefore: $$(t,\one)h\cdot ((s,\one)h)^{-1}=(t,\one)hh^{-1}(s,\one)^{-1}=(t+s,\one)\in G.$$ But since $E_k\subseteq E_{m-1}$, we know that $t+s\in E_m\setminus E_{m-1}$. By Proposition \ref{prop: final main0} we get that $E_m\times \{\one\}\subseteq H\cap K_n$, and hence $(t,\one)\in G$. But then $(t,\one)\cdot (t,\one)h=h\in G$, finishing the proof.
\end{proof}

\subsection{Counterexample} \label{Section: Revision: Counterexample}

In this subsection we give a counterexample to Conjecture \ref{Conj: Goksel}. The main idea goes as follows. If we have subgroups $H,G\leq \Tn$ such that $H$ contains an $n$-odometer and such that $H$ is elementwise $K_n$-conjugate into $G$, then up to conjugacy we can assume that $H\cap G$ contains the standard $n$-odometer $c|_{T_n}$ as in Section \ref{Section: Revision: Flag}.

Next we consider the results from Section \ref{Section: Revision: Application}. By Proposition \ref{prop: final main0} we know that there exists an integer $r$ such that $$H\cap K_n=E_r\times\{\one\},$$ and using Lemma \ref{Lem: e_2 in E_2} we know that $E_r=\langle e_1,...,e_r\rangle$, where $e_1,...,e_{2^{n-1}}$ are the elements defined in Definition \ref{Def: the functions e_m}. By Theorem \ref{IntroThm: counter} we know that $H$ is globally $K_n$-conjugate into $G$ if and only if $H\subseteq G$, and by Theorem \ref{Thm: final main2}, we know that if $H\not\subseteq G$, then for any $h\in H\setminus (H\cap G)$ there exists some integer $m$ such that $T_h\subseteq E_m\setminus E_{m-1}$.

These results make it easier to manually check examples. For every integer $m$, we let $f_m\in \Tnn^2\rtimes S_2$ denote the element that equals $(e_m,\one)\in \F_2^{V_{n-1}}\rtimes \Tnn$ as elements of $\Tn$.

\begin{example} \label{Ex: counterexample}
    Consider $n=4$. For every $1\leq m\leq 8$, let $$B_m:=\langle f_1,...,f_m\rangle \subseteq \aut(T_3)^2\rtimes S_2$$ denote the subgroup that equals $E_m\times \{\one\}\subseteq \F_2^{V_3}\rtimes \aut(T_3)$ as subgroups of $\aut(T_4)$.

    The counterexample to Conjecture \ref{Conj: Goksel} is given by $$H=\langle (c|_{T_3}^2,\one),c|_{T_4},B_3\rangle, \phantom{;;;} G=\langle f_4\cdot (c|_{T_3}^2,\one),c|_{T_4},B_3\rangle,$$ both identified with subgroups of $\aut(T_3)^2\rtimes S_2$. We show that $H$ is elementwise $K_4$-conjugate into $G$, but not globally $K_4$-conjugate into $G$, by using a brute-force SageMath program.

    The groups are coded as subgroups of $S_{16}=S_{2^4}$, by their permutations of $V_4$. Therefore, we have to describe a bijection $$\chi_2:V_4\isomto \{1,...,16\},$$ since in SageMath, $S_{16}$ is described as a permutation group on $\{1,...,16\}$, while in our labelling of $V_4=\{0,1\}^4$ defined in Section \ref{Section: Preliminaries: rooted binary trees explanation}, we start counting from 0. 
    
    We define the bijection as follows. For each word $v=v_0v_1v_2v_3\in V_4=\{0,1\}^4$, we let $v$ correspond to $$\chi_2(v):=1+\sum_{i=0}^3 v_i\cdot 2^i\in \{1,...,16\}.$$ Then, for each $h\in \aut(T_4)$, we code $$h(\chi_2(v)):=\chi_2(h(v)).$$ The resulting permutations are shown below. As for the elements of $K_4$, we use the recursion $$E_8=\langle e_1,e_2,e_3,e_4,e_5,e_6,e_7,e_8\rangle$$ for $e_m=(I+A_{c|_{T_4}})e_{m+1}$. The element $e_8:V_3\to \F_2\cong S_2$, which by definition has $e_8(0^3)=(0\ 1)\in S_2$ and $e_8(v)=1\in S_2$ for $v\neq 0^3$, corresponds to the transposition $(1\ 9)$, as it only permutes $\chi_2(0^4)=1$ with $\chi_2(0^31)=9$. Using (\ref{Eq: conjugating t by some g}), we find $$f_m=f_{m+1}(c|_{T_4})f_{m+1}(c|_{T_4})^{-1},$$ and $K_4=\langle f_1,f_2,f_3,f_4,f_5,f_6,f_7,f_8\rangle$. Note that in the program, we have to define the $f_m$ by instead reversing the word order, as SageMath uses right actions instead of left actions. The program is written as follows.

\begin{lstlisting}
S = SymmetricGroup(16);
g = S("(1,2,3,4,5,6,7,8,9,10,11,12,13,14,15,16)"); # Standard 4-odometer g
h = S("(1,5,9,13)(3,7,11,15)"); # (g^2,1)
f8 = S("(1,9)");
f7 = g^(-1)*f8*g*f8; # We don't take f8*g*f8*g^(-1), as sage
f6 = g^(-1)*f7*g*f7; # works with right actions.
f5 = g^(-1)*f6*g*f6;
f4 = g^(-1)*f5*g*f5;
f3 = g^(-1)*f4*g*f4;
f2 = g^(-1)*f3*g*f3;
f1 = g^(-1)*f2*g*f2;
K = S.subgroup([f8,f7,f6,f5,f4,f3,f2,f1]);
H = S.subgroup([h,g,f3,f2,f1]); # Group H defined as <h,g,B_3>
G = S.subgroup([g,f4*h,f3,f2,f1]); # Group G defined as <f_4*h,g,B_3>
eltwise = True;
for x in H: # Check for each x if it is K-conjugate into G
    for v in K:
        if v*x*v^-1 in G:
            break;
    else: # Hits the else statement only if x is not K-conjugate into G
        eltwise = False;
        print("H is not elementwise K-conjugate into G, because of:");
        print(x);
        break;
if(eltwise):
    print("H is elementwise K-conjugate into G.");
else:
    print("H is not elementwise K-conjugate into G.");
if(H == G):
    print("H = G.");
else:
    print("H =/= G.");
if(H.order() == G.order()):
    print("H and G have the same order.");
else:
    print("H and G do not have the same order.");
\end{lstlisting}

This program prints out that $|H|=|G|$, that $H$ is elementwise $K_n$-conjugate into $G$, and that $H\neq G$. Using theorem \ref{IntroThm: counter}, this implies that $H$ is not globally $K_n$-conjugate into $G$, which shows that this is indeed a counterexample to Conjecture \ref{Conj: Goksel}.

\end{example}

\section{Discussion} \label{Section: discussion}

As it happens, Example \ref{Ex: counterexample}, which is the counterexample to Conjecture \ref{Conj: Goksel}, shows something more. In adding the line \begin{lstlisting}
H.is_isomorphic(G)
\end{lstlisting} 

to the code, the program returns that $H$ and $G$ are not even isomorphic groups. As a result, something stronger occurs: elementwise $K_n$-conjugacy of $H\leq \Tn$ into $G\leq \Tn$ does not even necessarily imply that $H$ is isomorphic to a subgroup of $G$. 

If the properties $\p(G_n(f),M_n(f))$ do in fact always hold, then we've shown that a proof must rely on more than purely group theoretic properties. That is, the local to global conjugation must be forced by some property from either the Galois group or the even Markov group that has not yet been documented.

In fact, since the order of $H$ can be checked to be 128, the main result from \cite{Sonn} shows that $H$ is the Galois group of some polynomial over $\mathbb{Q}$. However, this does not imply that $H$ is the Galois group corresponding to a fourth iterate of a polynomial, nor that $G$ equals the Markov group with respect to that same polynomial. In any case, more work must be done in order to prove or disprove Conjecture \ref{Conj: GokselMarkov}.

\begin{acknowledgement}
    I would like to thank Pablo Lummerzheim, Jaime G{\'o}mez Ortiz, and my supervisor Olga Lukina for the many useful conversations related to the topic.
\end{acknowledgement}

\end{document}